\documentclass[11pt]{article}

\usepackage[margin=.85in]{geometry}  
\usepackage{graphicx}               
\usepackage{amsmath}              
\usepackage{amsfonts}              
\usepackage{bm}                        
\usepackage{amsthm}                
\usepackage{mathtools}             
\usepackage{esint}                     
\usepackage{color}
\usepackage{amssymb}             
\usepackage{hyperref}
\usepackage[noabbrev]{cleveref}
\usepackage{authblk}                 
\usepackage{tikz}                       
\usepackage{pgfplots}                

\newtheorem{thm}{Theorem}[section]
\newtheorem{lem}[thm]{Lemma}
\newtheorem{prop}[thm]{Proposition}
\newtheorem{cor}[thm]{Corollary}

\DeclareMathOperator{\loc}{loc}

\DeclareMathOperator{\reg}{reg}

\newcommand{\RR}{\mathbb{R}}     
\newcommand{\NN}{\mathbb{N}}     
\newcommand{\F}{\mathcal{F}}  
\newcommand{\J}{\mathcal{J}}  
\newcommand{\K}{\mathcal{K}}

\newcommand{\e}{\varepsilon}

\newcommand{\RNum}[1]{\uppercase\expandafter{\romannumeral #1\relax}}

\newcommand{\niii}[1]{{\left\vert\kern-0.25ex\left\vert\kern-0.25ex\left\vert #1 
    \right\vert\kern-0.25ex\right\vert\kern-0.25ex\right\vert}}

\numberwithin{equation}{section} 

\allowdisplaybreaks 

\begin{document}

\nocite{*}

\title{Higher order Gamma-limits for singularly perturbed Dirichlet-Neumann problems}

\author{
    Giovanni Gravina\\
    Department of Mathematical Sciences\\
    Carnegie Mellon University\\
    Pittsburgh, PA, USA\\
    ggravina@andrew.cmu.edu
  \and
    Giovanni Leoni\\
    Department of Mathematical Sciences\\
    Carnegie Mellon University\\
    Pittsburgh, PA, USA\\
    giovanni@andrew.cmu.edu
    }

\maketitle

\begin{abstract}
A mixed Dirichlet-Neumann problem is regularized with a family of singularly perturbed Neumann-Robin boundary problems, parametrized by $\e > 0$. Using an asymptotic development by Gamma-convergence, the asymptotic behavior of the solutions to the perturbed problems  is studied as $\e \to 0^+$, recovering classical results in the literature. 
\vspace*{10px}
\end{abstract}

 
\section{Introduction}
Mixed Dirichlet-Neumann boundary value problems arise naturally from a wide range of applications. Examples are the problem of a rigid punch or stamp making contact with an elastic body (see \cite{cd96}, \cite{MR1634297}, \cite{MR548943}, and the references therein), the steady flow of an ideal inviscid and incompressible fluid through an aperture in a reservoir (see \cite{MR3707494}, \cite{MR548943}, and the references therein), as well as free boundary problems (see, e.g., \cite{altcaffarelli81}).

The prototype for this kind of problems is given by 
\begin{equation}
\label{P0}
\left\{
\arraycolsep=1.4pt\def\arraystretch{1.6}
\begin{array}{rll}
\Delta u_0 = & f &  \text{ in } \Omega, \\
\partial_{\nu}u_0 = & 0 & \text{ on } \Gamma_N, \\
u_0 = & g & \text{ on } \Gamma_D,
\end{array}
\right.
\end{equation}
where $\Omega \subset \RR^N$ is an open set with sufficiently smooth boundary and $\Gamma_D,\Gamma_N$ are disjoint sets such that 
\[
\partial \Omega = \overline{\Gamma_D} \cup \overline{\Gamma_N}.
\]

It is well known (see \cite{dauge}, \cite{grisvard85}, \cite{kondratiev}, and \cite{MR0601607}) that solutions to mixed boundary problems are in general not smooth near the points on the boundary of the domain where two different conditions meet. Indeed, when $N = 2$ in (\ref{P0}), $f = 0$, $g = 0$, and $\Omega$ is given in polar coordinates by 
\[
\{(r,\theta) : r > 0, 0 < \theta < \pi\},
\]
the function $S \colon \Omega \to \RR$ given is polar coordinates by\footnote{In what follows, given a function $v = v(\bm{x})$ where $\bm{x} = (x,y)$, we denote by $\bar{v}$ the function $\bar{v}(r,\theta) \coloneqq v(r \cos \theta, r \sin \theta)$, and with a slight abuse of notation we write $v = \bar{v}(r,\theta)$.}
\begin{equation}
\label{singularfunction}
\bar{S}(r,\theta) \coloneqq r^{1/2}\sin\left(\theta/2\right)
\end{equation}
is a solution to (\ref{P0}). However, $S$ fails to be in $H^2$ in any neighborhood of the origin. 

In dimension $N = 2$ it turns out that functions of the type (\ref{singularfunction}) completely characterize the behavior of solutions to (\ref{P0}). Indeed, we have the following classical result (see \cite{dauge}, \cite{grisvard85}, \cite{kondratiev}, and \cite{MR0601607}).

\begin{thm}
\label{H2dec}
Let $N = 2$, and let $\Omega$ be an open, bounded, and connected subset of $\RR^2$, with $\partial \Omega$ of class $C^{1,1}$. Assume that $\Gamma_D$ and $\Gamma_N$ are nonempty, relatively open, and connected subsets of $\partial \Omega$ with 
\[
\partial \Omega = \overline{\Gamma_D} \cup \overline{\Gamma_N}, \quad \text{ and } \quad \overline{\Gamma_D} \cap \overline{\Gamma_N} = \{\bm{x}_1, \bm{x}_2\},
\]
and that $\partial \Omega \cap B_{\rho}(\bm{x}_i)$ is a segment for $i = 1,2$ and for some $0 < \rho < \min\{1, |\bm{x}_1 - \bm{x}_2|/2\}$. Let $f \in L^2(\Omega)$, $g \in H^{3/2}(\partial \Omega)$, and let $u \in H^1(\Omega)$ be a weak solution to $(\ref{P0})$. Then $u$ admits the decomposition
\[
u = u_{\reg} + \sum_{i = 1}^2c_iS_i,
\]  
where $u_{\reg} \in H^2(\Omega)$ and the $c_i$ are coefficients that only depend on $u$. The singular functions $S_i$ are given by the formula
\[
\bar{S}_i(r_i,\theta_i) = \bar{\varphi}(r_i)r_i^{1/2}\sin(\theta_i/2),
\]
where $(r_i, \theta_i)$ are polar coordinates centered at $\bm{x}_i$ such that 
\begin{align*}
\Omega \cap B_{\rho}(\bm{x}_i) = &\ \left\{\bm{x}_i + (r_i,\theta_i) : 0 < r_i < \rho, 0 < \theta_i < \pi \right\}, \\
\Gamma_D \cap B_{\rho}(\bm{x}_i) = &\ \left\{\bm{x}_i + (r_i,0) : 0 < r_i < \rho \right\},
\end{align*}
and $\bar{\varphi} \in C^{\infty}([0,\infty))$ is such that $\bar{\varphi} \equiv 1$ in $[0,\rho/2]$ and $\bar{\varphi} \equiv 0$ outside $[0,\rho]$. Furthermore, there exists a constant $c$, which only depends on the geometry of $\Omega$, such that
\[
\|u_{\reg}\|_{H^2(\Omega)} + \sum_{i = 1}^2|c_i| \le c\left(\|f\|_{L^2(\Omega)} + \|g\|_{H^{3/2}(\partial \Omega)}\right).
\]
\end{thm}

An approach that often proved to be successful for the study of ill-posed problems, and in general for problems that present singularities of some kind, is to consider a small perturbation, typically chosen with an opportunely regularizing effect, and then carry out a careful analysis on the convergence of solutions of the regularized problems to solutions of the original one. This procedure often requires to prove estimates that are independent of the parameter of the regularization. We refer to the classical monograph of Lions \cite{lions} for more details.

The aim of this paper is to regularize the problem (\ref{P0}) by introducing a family of mixed Neumann-Robin boundary value problems parametrized by $\e > 0$. To be precise, we consider
\begin{equation}
\label{Pe}
\left\{
\arraycolsep=1.4pt\def\arraystretch{1.6}
\begin{array}{rll}
\Delta u_{\e} =  & f & \text{ in } \Omega, \\
\partial_{\nu}u_{\e} = & 0 & \text{ on } \Gamma_N, \\
\e \partial_{\nu} u_{\e} + u_{\e} = & g & \text{ on } \Gamma_D.
\end{array}
\right.
\end{equation}

The convergence of solutions to (\ref{Pe}) to solutions of (\ref{P0}) has been studied by Costabel and Dauge in \cite{cd96} using classical PDE expansions (see \cite{lions}), who proved the following result.

\begin{thm}[Costabel-Dauge]
\label{cdthm}
Let $N = 2$, $\Omega$ be as in \Cref{H2dec}, $f = 0$, $g \in H^{1 + \delta}(\Gamma_D)$ for some $\delta > 0$, and let $u_{\e}$ and $u_0$ be solutions to $(\ref{Pe})$ and $(\ref{P0})$ (with $f = 0$), respectively. Then 
\begin{align}
\|u_{\e} - u_0\|_{L^2(\Omega)} = &\ \mathcal{O}(\e \log \e), \notag  \\
\|u_{\e} - u_0\|_{H^{1 + s}(\Omega)} = &\ \mathcal{O}(\e^{1/2 - s}), \text{ for } s \in \left(-\frac{1}{2}, \frac{1}{2}\right), \label{ue-u0H1+s}\\
\big\|(u_{\e} - u_0)_{|_{\Gamma_D}}\big\|_{L^2(\Gamma_D)} = &\ \mathcal{O}(\e\sqrt{|\log \e|}). \label{ue-u0L2bdry}
\end{align}
Moreover, these estimates cannot be improved in general.
\end{thm}
We refer to \cite{cd96} for the precise statement in the case $f \neq 0$. This problem was also previously considered by Colli Franzone in \cite{MR0417565}, where the author proved estimates on the difference $u_{\e} - u_0$ in certain Sobolev norms (see also the work of Aubin \cite{aubin} and Lions \cite{lions}). 

The question of convergence of solutions to the family of problems (\ref{Pe}) to the solution to (\ref{P0}) is of significance for the numerical approximations of (\ref{P0}). We refer to \cite{belgacem}, \cite{MR1115205}, \cite{MR1634297}, \cite{MR0448191}, \cite{MR0367754}, and the references therein for more information on this topic. 

In this paper we present an alternative proof of the estimates (\ref{ue-u0H1+s}) with $s = 0$ and (\ref{ue-u0L2bdry}) using the variational structure of (\ref{Pe}). Indeed, solutions to (\ref{Pe}) are minimizers of the functional 
\begin{equation}
\label{functionals}
\int_{\Omega}\left(\frac{1}{2}|\nabla v|^2 + fv\right)\,d\bm{x} + \frac{1}{2\e}\int_{\Gamma_D}(v - g)^2\,d\mathcal{H}^1, \quad v \in H^1(\Omega).
\end{equation}
Thus a natural approach is to use the notion of Gamma-convergence ($\Gamma$-convergence in what follows) introduced by De Giorgi in \cite{MR0375037} (for more information see also \cite{braides} and \cite{dalmaso}).

We recall that given a metric space $X$ and a family of functions $\F_{\e} \colon X \to \overline{\RR}$, $\e > 0$, we say that $\{\F_{\e}\}_{\e}$ $\Gamma$-converges to $\F_0 \colon X \to \overline{\RR}$ as $\e \to 0^+$, and we write $\F_{\e} \overset{\Gamma}{\to} \F_0$, if for every sequence $\e_n \to 0^+$ the following two conditions hold:
\begin{itemize}
\item[$(i)$] \emph{liminf inequality}: for every $x \in X$ and every sequence $\{x_n\}_n$ of elements of $X$ such that $x_n \to x$,
\[
\liminf_{n \to \infty}\F_{\e_n}(x_n) \ge \F_0(x);
\]
\item[$(ii)$] \emph{limsup inequality}: for every $x \in X$, there is a sequence $\{x_n\}_n$ of elements of $X$ such that $x_n \to x$ and 
\[
\limsup_{n \to \infty}\F_{\e_n}(x_n) \le \F_0(x).
\]
\end{itemize}
A sequence $\{x_n\}_n$ as in $(ii)$ is called a \emph{recovery sequence} for $x$. Moreover, we say that the \emph{asymptotic development by $\Gamma$-convergence of order $k$} 
\[
\F_{\e} = \F_0 + \omega_1(\e)\F_1 + \dots + \omega_k(\e)\F_k
\]
holds if there are functions $\F_i \colon X \to \overline{\RR}$, $i = 0, \dots k$, such that $\F_{\e} \overset{\Gamma}{\to} \F_0$ and for $i \ge 1$
\[
\F_{\e}^{(i)} \coloneqq \left(\F_{\e}^{(i - 1)} - \inf\{\F_{i - 1}(x) : x \in X\}\right)\frac{\omega_{i - 1}(\e)}{\omega_{i}(\e)} \overset{\Gamma}{\to} \F_{i},
\]
where $\F_{\e}^{(0)} \coloneqq \F_{\e}$, $\omega_{0} \equiv 1$ and for $i \ge 1$, $\omega_i \colon \RR^+ \to \RR^+$ is a suitably chosen function such that both $\omega_i$ and $\omega_i/\omega_{i-1}$ converge to zero as $\e \to 0^+$. We remark that for $\omega_i(\e) \coloneqq \e^i$ one has the standard power series asymptotic expansion 
\[
\F_{\e} = \F_0 + \e\F_1 + \dots + \e^k\F_k.
\]
We refer to \cite{MR1202527} and \cite{MR1373088} for more informations on asymptotic expansions by $\Gamma$-convergence. 

The powerfulness of asymptotic expansions by $\Gamma$-convergence has been shown in the recent papers \cite{MR3385194}, \cite{MR3448931}, \cite{leonimurray2}, and \cite{MR3568052}, where the authors completely characterized the second order asymptotic expansion of the Modica-Mortola functional and used it to obtain new important results on the slow motion of interfaces for the mass-preserving Allen-Cahn equation and the Cahn-Hilliard equation in higher dimensions.

In this paper we consider the $\Gamma$-convergence of the functionals (\ref{functionals}) with respect to convergence in $L^2(\Omega)$, and thus we define $\F_{\e} \colon L^2(\Omega) \to (-\infty, \infty]$ via
\begin{equation}
\label{Fe}
\F_{\e}(v) \coloneqq 
\left\{
\arraycolsep=1.4pt\def\arraystretch{1.6}
\begin{array}{ll}
\displaystyle\int_{\Omega}\left(\frac{1}{2}|\nabla v|^2 + fv\right)\,d\bm{x} + \frac{1}{2\e}\int_{\Gamma_D}(v - g)^2\,d\mathcal{H}^1 & \text{ if } v \in H^1(\Omega), \\
+\infty & \text{ otherwise}.
\end{array}
\right.
\end{equation}

We begin by studying the $\Gamma$-convergence of order zero of (\ref{Fe}).

\begin{thm}[$0$th order $\Gamma$-convergence]
\label{0gc}
Let $\Omega \subset \RR^N$ be an open, bounded, connected set with Lipschitz continuous boundary, and let $\Gamma_D \subset \partial \Omega$ be non-empty and relatively open. Assume that $f \in L^2(\Omega)$ and $g \in H^{1/2}(\partial \Omega)$. Then the family of functionals $\{\F_{\e}\}_{\e}$ defined in $(\ref{Fe})$ $\Gamma$-converges in $L^2(\Omega)$ to the functional 
\begin{equation}
\label{F0}
\F_0(v) \coloneqq 
\left\{
\arraycolsep=1.4pt\def\arraystretch{1.6}
\begin{array}{ll}
\displaystyle\int_{\Omega}\left(\frac{1}{2}|\nabla v|^2 + fv\right)\,d\bm{x} & \text{ if } v \in V, \\
+\infty & \text{ otherwise},
\end{array}
\right.
\end{equation}
where 
\begin{equation}
\label{Vspace}
V \coloneqq \{v \in H^1(\Omega) : v = g \text{ on } \Gamma_D\}.
\end{equation}
\end{thm}

Since the first asymptotic development by $\Gamma$-convergence of (\ref{Fe}) strongly relies on \Cref{H2dec}, in what follows we assume $N = 2$. We begin with a compactness result.
\begin{thm}[Compactness]
\label{1gc-cpt}
Let $N = 2$, $\Omega$ be as in \Cref{H2dec}, $f \in L^2(\Omega)$, $g \in H^{3/2}(\partial \Omega)$, $\F_{\e}$ and $\F_0$ be the functionals defined in $(\ref{Fe})$ and $(\ref{F0})$, respectively, and define
\begin{equation}
\label{Fe1}
\F_{\e}^{(1)} \coloneqq \frac{\F_{\e} - \min \F_0}{\e|\log \e|}.
\end{equation}
If $\e_n \to 0^+$ and $v_n \in L^2(\Omega)$ are such that 
\[
\sup \{\F_{\e_n}^{(1)}(v_n) : n \in \NN\} < \infty,
\]
then there exist a subsequence $\{v_{n_k}\}_k$ of $\{v_n\}_n$, $r_0 \in H^1(\Omega)$ and $v_0 \in L^2(\Gamma_D)$ such that
\begin{align}
\frac{v_{n_k} - u_0}{\sqrt{\e_{n_k}|\log \e_{n_k}|}} \rightharpoonup &\  r_0 \quad \text{ in } H^1(\Omega),  \label{1cpt1}\\
\frac{v_{n_k} - u_0}{\e_{n_k}\sqrt{|\log \e_{n_k}|}} \rightharpoonup &\ v_0 \quad \text{ in } L^2(\Gamma_D), \label{1cpt2}
\end{align}
where $u_0$ is the solution to $(\ref{P0})$.
\end{thm}

\begin{thm}[$1$st order $\Gamma$-convergence]
\label{1gc}
Under the assumptions of \Cref{1gc-cpt}, the family $\{\F_{\e}^{(1)}\}_{\e}$ $\Gamma$-converges in $L^2(\Omega)$ to the functional 
\begin{equation}
\label{F1}
\F_1(v) \coloneqq 
\left\{
\arraycolsep=1.4pt\def\arraystretch{1.6}
\begin{array}{ll}
-\displaystyle \frac{1}{8}\sum_{i = 1}^2c_i^2 & \text{ if } v = u_0, \\
+\infty & \text{ otherwise},
\end{array}
\right.
\end{equation}
where the coefficients $c_i = c_i(u_0)$ are as in \Cref{H2dec}. In particular, if $u_{\e} \in H^1(\Omega)$ is a solution to $(\ref{Pe})$, then 
\begin{equation}
\label{Fe1Tay}
\F_{\e}(u_{\e}) = \F_0(u_0) + \e |\log \e|\F_1(u_0) +  o\left(\e |\log \e|\right).
\end{equation}
\end{thm}

To characterize the second order asymptotic development by $\Gamma$-convergence of the family of functionals $\{\F_{\e}\}_{\e}$, we  introduce the auxiliary functional 
\begin{equation}
\begin{aligned}
\label{Ji}
\J_i(w) \coloneqq &\ \int_{\RR^2_+}|\nabla w(\bm{x})|^2\,d\bm{x} + \int_0^1\left(w(x,0)^2 - c_ix^{-1/2}w(x,0)\right)\,dx \\
&\ \quad + \int_1^{\infty}\left(w(x,0) - \frac{c_i}{2}x^{-1/2}\right)^2\,dx
\end{aligned}
\end{equation}
defined in
\begin{equation}
\label{spaceH}
H \coloneqq \{w \in H^1_{\loc}(\RR^2_+) : w \in H^1(B_R^+(\bm{0})) \text{ for every } R > 0\},
\end{equation}
where $w(\cdot,0)$ indicates the trace of $w$ on the positive real axis. Let\footnote{In what follows, given a function $v = v(\bm{x})$, we denote by $\bar{v}^{i}$ the function $\bar{v}^{(i)}(r_i,\theta_i) \coloneqq v(\bm{x}_i + r_i(\cos \theta_i, \sin \theta_i))$, for polar coordinates $(r_i,\theta_i)$ given as in \Cref{H2dec}.} 
\begin{align}
A_i \coloneqq &\ \inf\{\J_i(w) : w \in H\}, \label{A_i} \\
B_i \coloneqq &\ \frac{1}{2}\int_0^{\rho}\bar{\varphi}(r_i)r_i^{-1/2}\overline{\partial_{\nu}u_{\reg}^0}^{(i)}(r_i,0)\,dr_i, \label{B_i} \\
C_{\varphi} \coloneqq &\ \frac{1}{8}\int_{\rho/2}^1\left(1 - \bar{\varphi}(x)^2\right)x^{-1}\,dx, \label{C_i} \\
\bar{\psi}_i(r_i) \coloneqq &\ \frac{1}{2}\bar{\varphi}(r_i)r_i^{-1/2}. \label{psi_ie}
\end{align}
As shown in \Cref{auxVP}, there exists $w_i \in H$ such that $\J_{i}(w_i) = A_i$, and thus $w_i$ satisfies 
\begin{equation}
\label{aux}
\left\{
\arraycolsep=1.4pt\def\arraystretch{1.6}
\begin{array}{rll}
\Delta w_i =  & 0 & \text{ in } \RR^2_+, \\
\partial_{\nu} w_i = & 0 & \text{ on } (-\infty, 0) \times \{0\}, \\
\partial_{\nu} w_i + w_i = & \frac{c_i}{2}x^{-1/2} & \text{ on } (0,\infty) \times \{0\}.
\end{array}
\right.
\end{equation}
Observe that if $c_i = 0$ then $\J_i \ge 0$ and so $w_i = 0$ and $A_i = 0$. Finally, let $u_1 \in H^1(\Omega)$ be the solutions to the Dirichlet-Neumann problem 
\begin{equation}
\label{mixedu1}
\left\{
\arraycolsep=1.4pt\def\arraystretch{1.6}
\begin{array}{rll}
\Delta u_1 =  & 0 & \text{ in } \Omega, \\
\partial_{\nu}u_1 = & 0 & \text{ on } \Gamma_N, \\
u_1 = & -\partial_{\nu}u_{\reg}^0 & \text{ on } \Gamma_D.
\end{array}
\right.
\end{equation}

\begin{thm}[Compactness]
\label{2gc-cpt}
Let $N = 2$, $\Omega$ be as in \Cref{H2dec}, $f \in L^2(\Omega)$, $g \in H^{3/2}(\partial \Omega)$, $\F_{\e}$, $\F_0$, $\F_{\e}^{(1)}$, $\F_1$, $\J_i$ be as in $(\ref{Fe})$, $(\ref{F0})$, $(\ref{Fe1})$, $(\ref{F1})$, and $(\ref{Ji})$, respectively, and define
\begin{equation}
\label{Fe2}
\F_{\e}^{(2)} \coloneqq \frac{\F_{\e}^{(1)} - \min \F_1}{1/|\log \e|} = \frac{\F_{\e} - \min \F_0}{\e} - |\log \e|\min \F_1.
\end{equation}
If $\e_n \to 0^+$, $w_n \in L^2(\Omega)$ are such that 
\[
\sup \{\F_{\e_n}^{(2)}(w_n) : n \in \NN\} < \infty,
\]
and $W_{i,n} \in H$ is defined as 
\begin{equation}
\label{barw}
\bar{W}_{i,n}(r_i,\theta_i) \coloneqq \bar{\varphi}(r_i\e_n)\frac{\bar{w}_n^{(i)}(r_i\e_n,\theta_i) - \bar{u}^{(i)}_0(r_i\e_n,\theta_i) - \e_n\bar{u}^{(i)}_1(r_i\e_n, \theta_i)}{\sqrt{\e_n}}
\end{equation}
for $(r_i,\theta_i)$ polar coordinates as in \Cref{H2dec}, then there exist a subsequence $\{w_{n_k}\}_k$ of $\{w_n\}_n$, $w_0 \in H^1(\Omega)$ and $q_0 \in L^2_{\loc}(\Gamma_D)$ such that
\begin{align}
\frac{w_{n_k} - u_0 - \e_{n_k}u_1}{\sqrt{\e_{n_k}}} \rightharpoonup &\  w_0 \quad \text{ in } H^1(\Omega),  \label{2cpt1}\\
\frac{w_{n_k} - u_0}{\e_{n_k}} - u_1 - \sum_{i = 1}^2c_i \psi_i [1 - \chi_{B_{\e_{n_k}}(\bm{x}_i)}] \rightharpoonup &\ q_0 - \sum_{i = 1}^2c_i \psi_i \quad \, \text{ in } L^2(\Gamma_D),\label{2cpt2}
\end{align}
where $\psi_i$ is the function given in polar coordinates by $(\ref{psi_ie})$ and $u_1$ is the solution to $(\ref{mixedu1})$. Furthermore, for every $R > 0$,
\begin{equation}
\label{barw1}
W_{i,n_k} \rightharpoonup W_i \quad \text{ in } H^1(B_R^+(\bm{0})), \quad  \nabla W_{i,n_k} \rightharpoonup \nabla W_i \quad \text{ in } L^2(\RR^2_+;\RR^2)),\\
\end{equation}
\begin{align}
W_{i,n_k}(\cdot, 0) \rightharpoonup &\ W_i(\cdot, 0) \quad \quad \quad \quad \quad \ \, \text{ in } L^2((0,1) \times \{0\}), \label{barw2}\\
W_{i,n_k}(\cdot, 0) - \frac{c_i}{2}x^{-1/2} \rightharpoonup &\ W_i(\cdot, 0) - \frac{c_i}{2}x^{-1/2} \quad \text{ in } L^2((1,\infty) \times \{0\}), \label{barw3}
\end{align}
for some $W_i \in H$ such that $\J_i(W_i) < \infty$, where $W_{i,n_k}(\cdot,0)$ and $W_i(\cdot,0)$ indicate the trace of $W_{i,n_k}$ and $W_i$ on the positive real axis.
\end{thm}

\begin{thm}[$2$nd order $\Gamma$-convergence]
\label{2gc}
Under the assumptions of \Cref{2gc-cpt}, the family $\{\F_{\e}^{(2)}\}_{\e}$ $\Gamma$-converges in $L^2(\Omega)$ to the functional
\[
\label{F}
\F_2(v) \coloneqq 
\left\{
\arraycolsep=1.4pt\def\arraystretch{1.6}
\begin{array}{ll}
\sum_{i = 1}^2 \left( \frac{A_i}{2} + B_i c_i + C_{\varphi} c_i^2 \right) - \frac{1}{2} \int_{\Gamma_D}\left(\partial_{\nu}u_{\reg}^0\right)^2\,d\mathcal{H}^1 & \text{ if } v = u_0, \\
+\infty & \text{ otherwise},
\end{array}
\right.
\]
where the numbers $A_i, B_i$, and $C_{\varphi}$ are defined in $(\ref{A_i})$, $(\ref{B_i})$, and $(\ref{C_i})$, respectively. In particular, if $u_{\e} \in H^1(\Omega)$ is a solution to $(\ref{Pe})$ then 
\begin{equation}
\label{Fe2Tay}
\F_{\e}(u_{\e}) = \F_0(u_0) + \e |\log \e|\F_1(u_0) + \e \F_2(u_0) + o\left(\e\right).
\end{equation}
\end{thm}

As a consequence of our results, we obtain an alternative proof of the sharp estimates (\ref{ue-u0H1+s}) for $s = 0$ and (\ref{ue-u0L2bdry}) in \Cref{cdthm}. Indeed, we have the following theorem.

\begin{thm}
\label{mixedest}
Let $N = 2$, $\Omega$ as in \Cref{H2dec}, $f \in L^2(\Omega)$, $g \in H^{3/2}(\partial \Omega)$, and let $u_{\e}$ and $u_0$ be solutions to $(\ref{Pe})$ and $(\ref{P0})$, respectively. Then 
\begin{align}
\|u_{\e} - u_0\|_{L^2(\Gamma_D)} = &\ \mathcal{O}\left(\e \sqrt{|\log \e|}\right), \label{re-bdry} \\
\|\nabla (u_{\e} - u_0)\|_{L^2(\Omega;\RR^2)} = &\ \mathcal{O}\left(\e^{1/2}\right). \label{re-grad}
\end{align}
\end{thm}

In contrast to the work of Costabel and Dauge \cite{cd96}, our results rely on the variational structure of the mixed Neumann-Robin problem $(\ref{Pe})$, rather than the PDE. In particular, the compactness results in \Cref{1gc-cpt} and \Cref{2gc-cpt} are valid for energy bounded sequences and not just for minimizers, and thus are completely new. A key ingredient in the proof of compactness is the following Hardy-type inequality on balls due to Machihara, Ozawa and Wadade (see Corollary 6 in \cite{MR3102537}).
\begin{thm}
\label{hardyonballs}
Let $B_R(\bm{0})$ be the ball of $\RR^2$ with radius $R > 0$ and center at the origin. Then 
\[
\begin{aligned}
\left(\int_{B_R(\bm{0})}\,\frac{h(\bm{x})^2}{|\bm{x}|^2\left(1 + \log R - \log |\bm{x}|\right)^2}\,d\bm{x}\right)^{1/2} \le &\ \frac{\sqrt{2}}{R}\left(\int_{B_R(\bm{0})}h(\bm{x})^2\,d\bm{x}\right)^{1/2} \\
&\ \quad + 2(1 + \sqrt{2})\left(\int_{B_R(\bm{0})}\left|\frac{\bm{x}}{|\bm{x}|}\cdot \nabla h(\bm{x})\right|^2\,d\bm{x}\right)^{1/2}
\end{aligned}
\]
for every $h \in H^1(B_R(\bm{0}))$.
\end{thm}

It also important to observe that the asymptotic development by $\Gamma$-convergence leads naturally to the asymptotic expansion of the solutions $u_{\e}$ to $(\ref{Pe})$, and does not require an a priori ansatz of this expansion. Thus it could be applied to a large class of problems, including the $p$-Laplacian mixed problem
\[
\left\{
\arraycolsep=1.4pt\def\arraystretch{1.6}
\begin{array}{rll}
\operatorname{div}(|\nabla u_0|^{p - 2}\nabla u_0) = & f & \text{ in } \Omega, \\
|\nabla u_0|^{p-2}\partial_{\nu}u_0 = & 0 & \text{ on } \Gamma_N, \\
u_0 = & g & \text{ on } \Gamma_D.
\end{array}
\right.
\]
Another example is the seminal paper \cite{bcn90}, where  Berestycki, Caffarelli and Nirenberg considered the family of elliptic equations
\begin{equation}
\label{Lube}
L u_{\e} = \beta_{\e}(u_{\e})
\end{equation}
to approximate (as $\e \to 0^+$) a one-phase free boundary problem. Here the family $\{\beta_{\e}\}_{\e}$ is an approximate identity and the term $\beta_{\e}(u_{\e})$ is non-zero only for values of $u_{\e}$ less than $\e$. In particular, the region $\{u_{\e} < \e\}$ can be thought of as an approximation of the free boundary of the solution to the limiting problem. One-phase free boundary problems with mixed boundary conditions are strongly related to problems arising in fluid-dynamics (see \cite{GL18}). Our original motivation for this paper was the study of the regularized problem
\[
\left\{
\arraycolsep=1.4pt\def\arraystretch{1.6}
\begin{array}{rll}
\Delta u_{\e} = & \frac{1}{2}\beta_{\e}(u_{\e})Q^2 & \text{ in } \Omega, \\
\partial_{\nu}u_{\e} = & 0 & \text{ on } \Gamma_N, \\
\e \partial_{\nu} u_{\e} + u_{\e} = & g & \text{ on } \Gamma_D,
\end{array}
\right.
\]
where $\{\beta_{\e}\}_{\e}$ is a family of approximate identities as in (\ref{Lube}) and $Q$ is a nonnegative function in $L^2_{\loc}(\Omega)$. Solutions $u_{\e}$ of this problem converge to a solution $u$ of the one-phase free boundary problem 
\[
\left\{
\arraycolsep=1.4pt\def\arraystretch{1.6}
\begin{array}{rll}
\Delta u =  & 0 & \text{ in } \Omega, \\
u = & 0,\ |\nabla u| = Q & \text{ on } \Omega \cap \partial \{u > 0\}, \\
\partial_{\nu}u = & 0 & \text{ on } \Gamma_N, \\
u = & g & \text{ on } \Gamma_D.
\end{array}
\right.
\]
The asymptotic development by $\Gamma$-convergence of the corresponding family of functionals
\[
\int_{\Omega}\left(|\nabla v|^2 + B_{\e}(v)Q^2\right)\,d\bm{x} + \frac{1}{\e}\int_{\Gamma_D}(v - g)^2\,d\mathcal{H}^{N-1}, \quad v \in H^1(\Omega)
\]
is ongoing work. Here $B_{\e}$ is a primitive of $\beta_{\e}$.

Our paper is organized as follows: in Section 2 we study the minimization problem for the functional (\ref{Pe}) and prove \Cref{0gc}. As a consequence, in \Cref{convofmin} we show that there exists a unique variational solution to the problem (\ref{P0}). Section 3 is devoted to the study of the simpler case in which $\Gamma_D = \partial \Omega$, so that (\ref{Pe}) reduces to
\begin{equation}
\label{Re}
\left\{
\arraycolsep=1.4pt\def\arraystretch{1.6}
\begin{array}{rll}
\Delta u_{\e} = & f &  \text{ in } \Omega, \\
\e \partial_{\nu}u_{\e} + u_{\e} = & g & \text{ on } \partial \Omega.
\end{array}
\right.
\end{equation} 
Under suitable regularity assumptions on the set $\Omega$, we characterize the complete asymptotic expansion by $\Gamma$-convergence of $\{\F_{\e}\}_{\e}$, still defined as in (\ref{Fe}), but with $\Gamma_D$ replaced by $\partial \Omega$ (see \Cref{1gcJ}, \Cref{2gcJ}, and \Cref{kgcJ}). In \Cref{estimatesr1J} and \Cref{estimatesrkJ} we address the question of the convergence of $u_{\e}$ to $u_0$, i.e. the unique variational solution to the Dirichlet problem
\begin{equation}
\label{R0}
\left\{
\arraycolsep=1.4pt\def\arraystretch{1.6}
\begin{array}{rll}
\Delta u_0 = & f &  \text{ in } \Omega, \\
u_0 = & g & \text{ on } \partial \Omega.
\end{array}
\right.
\end{equation}
To be precise, we show that the asymptotic expansion 
\[
u_{\e} = \sum_{i = 1}^{\infty}\e^i u_i
\]
holds, where for every $i \in \NN$ the function $u_i$ is a solution to the Dirichlet problem
\[
\left\{
\arraycolsep=1.4pt\def\arraystretch{1.6}
\begin{array}{rll}
\Delta u_i = & 0 &  \text{ in } \Omega, \\
u_i = & -\partial_{\nu}u_{i - 1} & \text{ on } \partial \Omega.
\end{array}
\right.
\]
We remark that \Cref{estimatesrkJ} fully recovers the results of Theorem 2.3 in \cite{cd96} and that the auxiliary problems for $u_i$ arise naturally during the study of higher order $\Gamma$-limits of $\F_{\e}$ (see, for example, the proof of \Cref{2gcJ}). The case of a Robin boundary condition that transforms into a Dirichlet boundary condition for Helmholtz equation was considered by Kirsch in \cite{MR803839}. In Section 4 we prove our main results. In Section 5 we recast these results in a more general framework by decoupling the different scales in the asymptotic expansion of $u_{\e}$.
\section{Gamma-convergence of order zero and global minimizers}
Throughout the section we study the mixed problem (\ref{Pe}) and the associated minimization problem under the following assumptions on the set $\Omega$ and on $\Gamma_D$, namely the portion of the boundary where the Robin boundary condition is imposed:
\begin{equation}
\label{H0}\tag{$H_0$}
\left\{
\arraycolsep=1.4pt\def\arraystretch{1.6}
\begin{array}{rl}
(i) & \Omega \text{ is an open, bounded and connected subset of } \RR^N, \\
(ii) & \partial \Omega \text{ is Lipschitz continuous}, \\
(iii) & \Gamma_D \text{ is a non-empty and relatively open subset of } \partial \Omega.
\end{array}
\right.
\end{equation}
Furthermore, define $\Gamma_N \coloneqq \partial \Omega \setminus \overline{\Gamma_D}$. Notice that for the purposes of this section we do not assume that $\Gamma_N \neq \emptyset$; analogous results hold (with trivial changes) if $\Gamma_N = \emptyset$.

\begin{thm}
\label{existenceFe}
Let $\Omega$ be as in $($\ref{H0}$)$, $f \in L^2(\Omega)$, $g \in L^2(\partial \Omega)$, and $\e \in (0,1)$. Then the functional $\F_{\e}$, defined as in $(\ref{Fe})$, admits a unique minimizer $u_{\e} \in H^1(\Omega)$. Furthermore, $u_{\e}$ is a weak solution to the mixed Neumann-Robin problem $(\ref{Pe})$.
\end{thm}

The proof of \Cref{existenceFe} is based on the following well-known result.

\begin{lem}
\label{equivnorm}
Let $\Omega$ be as in $($\ref{H0}$)$ and for $u \in H^1(\Omega)$ set
\begin{equation}
\label{notationnorm}
\niii{u}_{H^1(\Omega)} \coloneqq \left(\|\nabla u\|^2_{L^2(\Omega;\RR^N)} + \|u\|^2_{L^2(\Gamma_D)}\right)^{1/2}.
\end{equation}
Then $\niii{ \cdot }_{H^1(\Omega)}$ defines a norm on $H^1(\Omega)$ that is equivalent to the standard norm, i.e., there are two constants $\kappa_1, \kappa_2$, which only depend on the geometry of $\Omega$ and $\Gamma_D$, such that for every $u \in H^1(\Omega)$,
\[
\kappa_1\niii{u}_{H^1(\Omega)} \le \|u\|_{H^1(\Omega)} \le \kappa_2\niii{u}_{H^1(\Omega)}.
\]
\end{lem}

\begin{proof}[Proof of \Cref{existenceFe}]
By H\"older's inequality, we have that for every $\e \in (0,1)$ and for every $u \in H^1(\Omega)$,
\begin{equation}
\label{energylowerbound}
\F_{\e}(u) \ge \frac{1}{2}\|\nabla u\|^2_{L^2(\Omega;\RR^N)} - \|f\|_{L^2(\Omega)}\|u\|_{L^2(\Omega)} + \frac{1}{2}\|u - g\|^2_{L^2(\Gamma_D)}.
\end{equation}
Young's inequality then implies
\begin{equation}
\label{u-g}
\|u - g\|_{L^2(\Gamma_D)}^2 = \|u\|_{L^2(\Gamma_D)}^2 + \|g\|_{L^2(\Gamma_D)}^2 -2\int_{\Gamma_D}ug\,d\mathcal{H}^{N-1} \ge \frac{1}{2} \|u\|_{L^2(\Gamma_D)}^2 - 7\|g\|_{L^2(\Gamma_D)}^2,
\end{equation}
and thus, combining the estimates (\ref{energylowerbound}) and (\ref{u-g}) with \Cref{equivnorm}, we obtain
\[
\F_{\e}(u) \ge \frac{1}{4}\niii{u}_{H^1(\Omega)}^2 - \kappa_2\|f\|_{L^2(\Omega)}\niii{u}_{H^1(\Omega)} - \frac{7}{2}\|g\|_{L^2(\Gamma_D)}^2.
\]
In turn, 
\[
\inf\{\F_{\e}(u) : \e \in (0,1), u \in L^2(\Omega)\} > -\infty
\]
and for every $\e \in (0,1)$ the functional $\F_{\e}$ is coercive. Since $\F_{\e}$ is lower semicontinuous with respect to weak convergence in $L^2(\Omega)$, the existence of a global minimizer $u_{\e}$ follows from the direct method in the calculus of variations and the assertion about uniqueness is a consequence of the strict convexity of the functional $\F_{\e}$. Moreover, one can check that $u_{\e}$ is a weak solution to (\ref{Pe}) by considering variations of the functional $\F_{\e}$. We omit the details.
\end{proof}

\begin{prop}[Compactness]
\label{compactness}
Under the assumptions of Theorem \ref{0gc}, if $\e_n \to 0^+$ and $u_n$ are such that 
\[
\sup\{\F_{\e_n}(u_n) : n \in \NN\} < \infty,
\]
then there exist a subsequence $\{u_{n_k}\}_{k}$ of $\{u_n\}_n$, $u \in V$ and $v \in L^2(\Gamma_D)$ such that  
\begin{align*}
u_{n_k} \rightharpoonup &\ u \quad  \text{ in } H^1(\Omega), \\
\e_{n_k}^{-1/2}(u_{n_k} - g) \rightharpoonup &\ v \quad \text{ in } L^2(\Gamma_D).
\end{align*}
\end{prop}
\begin{proof}
Let $M \coloneqq \sup_n\F_{\e_n}(u_n)$ and assume without loss of generality that $\e_1 \le 1$. Reasoning as in the proof of \Cref{existenceFe}, by H\"older's inequality we see that
\begin{equation}
\label{cptestimate}
M \ge \frac{1}{2}\|\nabla u_n\|_{L^2(\Omega;\RR^N)}^2 - \|f\|_{L^2(\Omega)}\|u_n\|_{L^2(\Omega)} + \frac{1}{2\e_n}\|u_n - g\|^2_{L^2(\Gamma_D)}
\end{equation}
for every $n \in \NN$. Young's inequality, together with the fact that $\e_n \le 1$, then implies that
\begin{equation}
\begin{aligned}
\label{un-g}
\frac{1}{2\e_n}\|u_n - g\|^2_{L^2(\Gamma_D)} \ge &\ \frac{1}{4}\|u_n - g\|^2_{L^2(\Gamma_D)} + \frac{1}{4\e_n}\|u_n - g\|^2_{L^2(\Gamma_D)} \\
\ge &\ \frac{1}{8}\|u_n\|_{L^2(\Gamma_D)}^2 - \frac{7}{4}\|g\|_{L^2(\Gamma_D)}^2 + \frac{1}{4\e_n}\|u_n - g\|^2_{L^2(\Gamma_D)},
\end{aligned}
\end{equation}
and thus, combining the estimates (\ref{cptestimate}) and (\ref{un-g}) with \Cref{equivnorm}, and using the notation (\ref{notationnorm}), we arrive at
\[
M \ge \frac{1}{8}\niii{u_n}_{H^1(\Omega)}^2 - \kappa_2\|f\|_{L^2(\Omega)}\niii{u_n}_{H^1(\Omega)} - \frac{7}{4}\|g\|_{L^2(\Gamma_D)}^2 + \frac{1}{4\e_n}\|u_n - g\|^2_{L^2(\Gamma_D)}.
\]
Consequently, $\{u_n\}_n$ is bounded in $H^1(\Omega)$ by \Cref{equivnorm}, and furthermore $\{\e_n^{-1/2}(u_n - g)\}_n$ is bounded in $L^2(\Gamma_D)$. Hence there are a functions $u \in H^1(\Omega)$, $v \in L^2(\Gamma_D)$ and a subsequence $\{u_{n_k}\}_k$ of $\{u_n\}_n$ as in the statement. To conclude we notice that $u_n \to g$ in $L^2(\Gamma_D)$, and so $u \in V$.
\end{proof}

\begin{proof}[Proof of \Cref{0gc}]
Let $\e_n \to 0^+$ and $\{u_n\}_n$ be a sequence of functions in $L^2(\Omega)$ such that $u_n \to u$ in $L^2(\Omega)$. If $\liminf_{n \to \infty} \F_{\e_n}(u_n) = \infty$ there is nothing to prove. Hence, up to the extraction of a subsequence (not relabeled), we can assume without loss of generality that 
\[
\liminf_{n \to \infty} \F_{\e_n}(u_n) = \lim_{n \to \infty} \F_{\e_n}(u_n) < \infty.
\]
In particular, $\F_{\e_n}(u_n) < \infty$ for every $n$ sufficiently large. Let $\{u_{n_k}\}_k$ and $u$ be given as in \Cref{compactness}, then 
\begin{align*}
\liminf_{k \to \infty} \F_{\e_{n_k}}(u_{n_k}) \ge \liminf_{k \to \infty} \int_{\Omega}\left(\frac{1}{2}|\nabla u_{n_k}|^2 + fu_{n_k}\right)\,d\bm{x}
\ge \frac{1}{2}\int_{\Omega}|\nabla u|^2\,d\bm{x} + \int_{\Omega}fu\,d\bm{x} = \F_0(u).
\end{align*}

On the other hand, for every $u \in L^2(\Omega)$, the constant sequence $u_n = u$ is a recovery sequence. Indeed, $\F_{\e_n}(u) = \F_0(u)$ for every $u \in V$, while if $u \notin V$ then $\F_0(u) = \infty$ and hence there is nothing to prove.
\end{proof}

\begin{cor}
\label{convofmin}
Under the assumptions of Theorem \ref{0gc}, if $\e_n \to 0^+$ and $\{u_n\}_n$ is a sequence of functions in $L^2(\Omega)$ such that 
\[
\limsup_{n \to \infty}\F_{\e_n}(u_n) \le \inf\left\{\F_0(v) : v \in L^2(\Omega)\right\}
\] 
then $u_n \to u_0$ strongly in $H^1(\Omega)$, where $u_0$ is the unique global minimizer of $\F_0$. In particular, global minimizers $u_{\e_n}$ of $\F_{\e_n}$ converge in $H^1(\Omega)$ to $u_0$.
\end{cor}
\begin{proof}
Since $g \in H^{1/2}(\partial \Omega)$, by standard trace theorems (see Theorem 18.40 in \cite{leoni}) the space $V$ defined in (\ref{Vspace}) is nonempty. In turn, the strictly convex functional $\F_0$ given in (\ref{F0}) admits a unique minimizer $u_0$ which is a weak solution to (\ref{P0}). Let $\{u_n\}_n$ be a sequence of functions in $L^2(\Omega)$ such that 
\begin{equation}
\label{limsup2.4}
\limsup_{n \to \infty} \F_{\e_n}(u_n) \le \F_0(u_0).
\end{equation}
Given a subsequence $\{\e_{n_k}\}_k$ of $\{\e_n\}_n$, by \Cref{compactness} we can find a further subsequence $\{u_{n_{k_j}}\}_j$ and $v_0 \in V$ such that $u_{n_{k_j}} \to v_0$. By $\Gamma$-convergence 
\[
\F_0(u_0) \ge \limsup_{j \to \infty} \F_{\e_{n_{k_j}}}(u_{n_{k_j}}) \ge F_0(v_0),
\]
which in turn implies that $v_0 = u_0$. Hence the full sequence $\{u_n\}_n$ converges in $L^2(\Omega)$ to $u_0$. Moreover, by (\ref{limsup2.4})
\begin{align*}
\F_0(u_0) \ge \limsup_{n \to \infty} \F_{\e_n}(u_n) \ge &\  \limsup_{n \to \infty} \int_{\Omega} \left(\frac{1}{2}|\nabla u_n|^2 + fu_n\right)\,d\bm{x} \\
\ge &\ \liminf_{n \to \infty}\frac{1}{2}\int_{\Omega}|\nabla u_n|^2\,d\bm{x} + \int_{\Omega}fu_0\,d\bm{x} \ge \F_0(u_0),
\end{align*}
and so
\[
\lim_{n \to \infty}\int_{\Omega}|\nabla u_n|^2\,d\bm{x}  = \int_{\Omega}|\nabla u_0|^2\,d\bm{x}.
\]
By the strict convexity of the $L^2$-norm it follows that $\nabla u_n \to \nabla u_0$ in $L^2(\Omega; \RR^2)$.
\end{proof}
\section{A problem without singularities}
Following Costabel and Dauge \cite{cd96}, in this section we will be concerned with the study of the easier case of the non-mixed problem (\ref{Re}); to be precise, it is assumed throughout the section that $\Gamma_D = \partial \Omega$. Under this additional assumption we prove asymptotic developments by $\Gamma$-convergence of all orders for the family of functionals $\{\F_{\e}\}_{\e}$ and deduce a complete asymptotic expansion for $u_{\e}$, i.e the solution to (\ref{Re}) (see \Cref{existenceFe}). Throughout the section, we will make the following assumptions on the set $\Omega$: 
\begin{equation}
\label{Hj}\tag{$H_j$}
\left\{
\arraycolsep=1.4pt\def\arraystretch{1.6}
\begin{array}{rl}
(i) & \Omega \text{ is an open, bounded and connected subset of } \RR^N, \\
(ii) & \partial \Omega \text{ is of class } C^{j,1}.  
\end{array}
\right.
\end{equation}

\subsection{The non-mixed problem: Gamma-convergence of order one}
In this section we prove a first order asymptotic expansion for $\F_{\e}$. We begin by studying the compactness properties of sequences with bounded energy.
\begin{prop}[Compactness]
\label{1compactnessJ}
Let $\Omega$ be as in $(H_1)$, $f \in L^2(\Omega)$, $g \in H^{3/2}(\partial \Omega)$, $\F_{\e}$ and $\F_0$ be the functionals defined in $(\ref{Fe})$ and $(\ref{F0})$ (with $\Gamma_D = \partial \Omega$), respectively, and define
\begin{equation}
\label{Fe1J}
\F_{\e}^{(1)} \coloneqq \frac{\F_{\e} - \min\F_0}{\e}.
\end{equation}
If $\e_n \to 0^+$ and $v_n \in L^2(\Omega)$ are such that 
\[
\sup\{\F_{\e_n}^{(1)}(v_n) : n \in \NN\} < \infty,
\]
then $u_n \to u_0$ in $H^1(\Omega)$ and there exist a subsequence $\{v_{n_k}\}_k$ of $\{v_n\}_n$, $r_0 \in H^1(\Omega)$ and $v_0 \in L^2(\partial \Omega)$ such that
\begin{equation}
\begin{aligned}
\label{1cptJ}
\frac{v_{n_k} - u_0}{\sqrt{\e_{n_k}}} \rightharpoonup &\ r_0 \quad \text{ in } H^1(\Omega), \\
\frac{v_{n_k} - u_0}{\e_{n_k}} \rightharpoonup &\ v_0 \quad \text{ in } L^2(\partial \Omega),
\end{aligned}
\end{equation}
where $u_0$ is the solution to $(\ref{R0})$.
\end{prop}

\begin{proof}
If we let $M \coloneqq \sup\{\F_{\e_n}^{(1)}(v_n) : n \in \NN\}$, then $\F_{\e}(v_n) \le \F_0(u_0) + \e_nM$. On the other hand, 
\[
\liminf_{n \to \infty}\F_{\e_n}(v_n) \ge \F_0(u_0)
\]
by \Cref{0gc}, and in turn $v_n \to u_0$ strongly in $H^1(\Omega)$ by \Cref{convofmin}.

For every $n \in \NN$, let $r_n \in L^2(\Omega)$ be such that $v_n = u_0 + \e_n r_n$. Then $\F_{\e_n}^{(1)}(v_n)$ can be rewritten as
\begin{equation}
\label{F1evnJ}
\F_{\e_n}^{(1)}(v_n) = \int_{\Omega}\left(\nabla u_0 \cdot \nabla r_n + \frac{\e_n}{2}|\nabla r_n|^2 + fr_n\right)\, d\bm{x} + \frac{1}{2}\int_{\partial \Omega}r_n^2\,d\mathcal{H}^{N-1}.
\end{equation}
Since $\partial \Omega$ is of class $C^{1,1}$, $f \in L^2(\Omega)$, and $g \in H^{3/2}(\partial \Omega)$, by standard elliptic regularity theory for (\ref{R0}), $u_0 \in H^2(\Omega)$ (see Theorem 2.4.2.5 in \cite{grisvard85}) and by an application of the divergence theorem we have
\begin{equation}
\label{divthmF1eJ}
\int_{\Omega}\left(\nabla u_0 \cdot \nabla r_n + f r_n \right)\,d\bm{x} = \int_{\partial \Omega}\partial_{\nu}u_0 r_n\,d\mathcal{H}^{N-1}.
\end{equation}
Substituting (\ref{divthmF1eJ}) into (\ref{F1evnJ}) we arrive at
\begin{equation}
\label{Fe1vn}
\begin{aligned}
M \ge \F_{\e_n}^{(1)}(v_n) = &\ \frac{\e_n}{2}\int_{\Omega}|\nabla r_n|^2\,d\bm{x} + \int_{\partial \Omega}\left(\frac{1}{2}r_n^2 + \partial_{\nu}u_0 r_n\right)\,d\mathcal{H}^{N-1} \\
= &\ \frac{\e_n}{2}\int_{\Omega}|\nabla r_n|^2\,d\bm{x} + \frac{1}{2}\int_{\partial \Omega}\left(r_n + \partial_{\nu}u_0\right)^2\,d\mathcal{H}^{N-1} - \frac{1}{2}\int_{\partial \Omega}\left(\partial_{\nu}u_0\right)^2\,d\mathcal{H}^{N-1},
\end{aligned}
\end{equation}
and (\ref{1cptJ}) is proved at once.
\end{proof}

\begin{thm}[$1$st order $\Gamma$-convergence]
\label{1gcJ}
Under the assumptions of \Cref{1compactnessJ}, the family $\{\F_{\e}^{(1)}\}_{\e}$ $\Gamma$-converges in $L^2(\Omega)$ to the functional
\begin{equation}
\label{F1J}
\F_1(v) \coloneqq 
\left\{
\arraycolsep=1.4pt\def\arraystretch{1.6}
\begin{array}{ll}
\displaystyle-\frac{1}{2}\int_{\partial \Omega}\left(\partial_{\nu}u_0\right)^2\,d\mathcal{H}^{N-1} & \text{ if } v = u_0, \\
+ \infty & \text{ otherwise}.
\end{array}
\right.
\end{equation}
In particular, if $u_{\e} \in H^1(\Omega)$ is the solution to $(\ref{Re})$, then 
\begin{equation}
\label{Fe1JTay}
\F_{\e}(u_{\e}) = \F_0(u_0) + \e \F_1(u_0) +  o\left(\e\right).
\end{equation}
\end{thm}

\begin{proof}
Let $\e_n \to 0^+$ and $\{v_n\}_n$ be a sequence of functions in $L^2(\Omega)$ such that $v_n \to v$ in $L^2(\Omega)$. Reasoning as in the proof of \Cref{0gc}, we can assume without loss of generality that 
\[
\liminf_{n \to \infty} \F_{\e_n}^{(1)}(v_n) = \lim_{n \to \infty} \F_{\e_n}^{(1)}(v_n) < \infty.
\]
In particular, $\F_{\e_n}^{(1)}(v_n) < \infty$ for every $n$ sufficiently large. Let $\{v_{n_k}\}_k$ be as in \Cref{1compactnessJ}. Then $v_n \to u_0$ in $H^1(\Omega)$ and from (\ref{Fe1vn}) we deduce that
\[
\liminf_{n \to \infty}\F_{\e_n}^{(1)}(v_n) \ge - \frac{1}{2}\int_{\partial \Omega}\left(\partial_{\nu}u_0\right)^2\,d\mathcal{H}^{N-1} = \F_1(u_0).
\]

On the other hand, for every $v \in L^2(\Omega) \setminus \{u_0\}$ the constant sequence $v_n = v$ is a recovery sequence. If now $v = u_0$, since by assumption $\partial_{\nu}u_0 \in H^{1/2}(\partial \Omega)$, we can find $w \in H^1(\Omega)$ such that $w = -\partial_{\nu}u_0$ on $\partial \Omega$, where the equality holds in the sense of traces. Set $v_n \coloneqq u_0 + \e_n w$. Then $v_n \to u_0$ in $H^1(\Omega)$ and again from (\ref{Fe1vn}) it follows that
\[
\lim_{n \to \infty}\F_{\e_n}^{(1)}(v_n) = \lim_{n \to \infty}\frac{\e_n}{2}\int_{\Omega}|\nabla w|^2\,d\bm{x} - \frac{1}{2}\int_{\partial \Omega}\left(\partial_{\nu}u_0\right)^2\,d\mathcal{H}^{N-1} = \F_1(u_0).
\]
This concludes the proof of the $\Gamma$-convergence. The energy expansion (\ref{Fe1JTay}) follows from Theorem 1.2 in \cite{MR1202527}.
\end{proof}

\subsection{The non-mixed problem: Gamma-convergence of order two}
In this section we prove a second order asymptotic expansion for $\F_{\e}$. As customary, we begin by investigating the compactness properties of sequences with bounded energy.
\begin{prop}[Compactness]
\label{2compactnessJ}
Let $\Omega$ be as in $(H_1)$, $f \in L^2(\Omega)$, $g \in H^{3/2}(\partial \Omega)$, $\F_{\e}$, $\F_0$, $\F_{\e}^{(1)}$, and $\F_1$ be as in $(\ref{Fe})$, $(\ref{F0})$, $(\ref{Fe1J})$, and $(\ref{F1J})$, respectively, and define
\[
\F_{\e}^{(2)} \coloneqq \frac{\F_{\e}^{(1)} - \min\F_1}{\e} = \frac{\F_{\e} - \min \F_0 - \e \min \F_1}{\e^2}.
\]
If $\e_n \to 0^+$ and $w_n \in L^2(\Omega)$ are such that 
\[
\sup\{\F_{\e_n}^{(2)}(w_n) : n \in \NN\} < \infty,
\]
then $w_n \to u_0$ in $H^1(\Omega)$ and there exist a subsequence $\{w_{n_k}\}_k$ of $\{w_n\}_n$, $w_0 \in H^1(\Omega)$ and $q_0 \in L^2(\partial \Omega)$ such that   
\begin{align*} 
\frac{w_{n_k} - u_0}{\e_{n_k}} \rightharpoonup &\ w_0 \quad \text{ in } H^1(\Omega), \\
\frac{w_{n_k} - u_0 + \e_{n_k}\partial_{\nu}u_0}{\e_{n_k}^{3/2}} \rightharpoonup &\ q_0 \quad \, \text{ in } L^2(\partial \Omega),
\end{align*}
where $u_0$ is the solution to $(\ref{R0})$. In particular, $w_0 = -\partial_{\nu}u_0$ on $\partial \Omega$ in the sense of traces. 
\end{prop}

\begin{proof}
By \Cref{convofmin}, we deduce that $w_n \to u_0$ in $H^1(\Omega)$. For every $n \in \NN$, let $r_n \in L^2(\Omega)$ be such that $w_n = u_0 + \e_n r_n$. Then $\F_{\e_n}^{(2)}(w_n)$ can be rewritten as
\begin{equation}
\label{Fe2vn}
\F_{\e_n}^{(2)}(w_n) = \frac{1}{2}\int_{\Omega}|\nabla r_n|^2\,d\bm{x} + \frac{1}{2\e_n}\int_{\partial \Omega}(r_n + \partial_{\nu}u_0)^2\,d\mathcal{H}^{N-1}.
\end{equation}
We then proceed as in the proof of \Cref{compactness} with $f = 0$, $g = - \partial_{\nu}u_0$ and $r_n$ in place of $u_n$.
\end{proof}

\begin{thm}[$2$nd order $\Gamma$-convergence]
\label{2gcJ}
Under the assumptions of \Cref{2compactnessJ}, let $u_1 \in H^1(\Omega)$ be the unique solution to the Dirichlet problem
\[
\left\{
\arraycolsep=1.4pt\def\arraystretch{1.6}
\begin{array}{rll}
\Delta u_1 = & 0 &  \text{ in } \Omega, \\
u_1 = & -\partial_{\nu}u_0 & \text{ on } \partial \Omega.
\end{array}
\right.
\]
Then the family $\{\F_{\e}^{(2)}\}_{\e}$ $\Gamma$-converges in $L^2(\Omega)$ to the functional 
\[ 
\F_2(v) \coloneqq 
\left\{
\arraycolsep=1.4pt\def\arraystretch{1.6}
\begin{array}{ll}
\displaystyle \frac{1}{2}\int_{\Omega}|\nabla u_1|^2\,d\bm{x} & \text{ if } v = u_0, \\
+\infty & \text{ otherwise}.
\end{array}
\right.
\]
In particular, if $u_{\e} \in H^1(\Omega)$ is the solution to $(\ref{Re})$, then 
\begin{equation}
\label{Fe2JTay}
\F_{\e}(u_{\e}) = \F_0(u_0) + \e \F_1(u_0) +  \e^2 \F_2(u_0) + o\left(\e^2\right).
\end{equation}
\end{thm}

\begin{proof}
Let $\e_n \to 0^+$ and $\{w_n\}_n$ be a sequence of functions in $L^2(\Omega)$ such that $w_n \to w$ in $L^2(\Omega)$. Reasoning as in the proof of \Cref{0gc}, we can assume without loss of generality that 
\[
\liminf_{n \to \infty} \F_{\e_n}^{(2)}(w_n) = \lim_{n \to \infty} \F_{\e_n}^{(2)}(w_n) < \infty.
\]
In particular, $\F_{\e_n}^{(2)}(w_n) < \infty$ for every $n$ sufficiently large. Let $\{w_{n_k}\}_k$ and $w_0$ be as in \Cref{2compactnessJ}. Then $w_n \to u_0$ in $H^1(\Omega)$ and from (\ref{Fe2vn}) we deduce that
\[
\begin{aligned}
\liminf_{k \to \infty}\F_{\e_{n_k}}^{(2)}(w_{n_k}) \ge &\ \liminf_{k \to \infty} \frac{1}{2}\int_{\Omega}|\nabla r_{n_k}|^2\,d\bm{x} \ge \frac{1}{2}\int_{\Omega}|\nabla w_0|^2\,d\bm{x} \\
\ge &\ \inf \left\{ \frac{1}{2}\int_{\Omega}|\nabla p|^2\,d\bm{x} : p \in H^1(\Omega),\ p = - \partial_{\nu}u_0 \text{ on } \partial \Omega \right\} = \frac{1}{2}\int_{\Omega}|\nabla u_1|^2\,d\bm{x} = \F_2(u_0).
\end{aligned}
\]
We remark that the function $u_1$ exists (and is unique) by an application of \Cref{convofmin}.

On the other hand, for every $w \in L^2(\Omega) \setminus \{u_0\}$ the constant sequence $w_n = w$ is a recovery sequence. As one can  check from (\ref{Fe2vn}), $w_n \coloneqq u_0 + \e_n u_1$ is a recovery sequence for $u_0$. This concludes the proof of the $\Gamma$-convergence. The energy expansion (\ref{Fe2JTay}) follows from Theorem 1.2 in \cite{MR1202527}.
\end{proof}

\begin{cor}
\label{estimatesr1J}
Let $\Omega$ be as in $(H_1)$, $f \in L^2(\Omega)$, $g \in H^{3/2}(\partial \Omega)$, and let $u_{\e}$ and $u_0$ be solutions to $(\ref{Re})$ and $(\ref{R0})$, respectively. Then there exists a constant $c > 0$, independent of $\e$, such that 
\[
\label{estimatereps}
\begin{aligned}
\|u_{\e} - u_0\|_{H^1(\Omega)} \le &\ c\e\left(\|f\|_{L^2(\Omega)} + \|g\|_{H^{3/2}(\partial \Omega)}\right), \\
\left\|u_{\e} - u_0 + \e \partial_{\nu}u_0\right\|_{L^2(\partial \Omega)} \le &\ c\e^{3/2}\left(\|f\|_{L^2(\Omega)} + \|g\|_{H^{3/2}(\partial \Omega)}\right).
\end{aligned}
\]
\end{cor}

\begin{proof}
If we let $w_{\e} \coloneqq u_0 + \e u_1$, for $u_1$ as in \Cref{2gcJ}, then
\[
\F_{\e}(w_{\e}) = \F_0(u_0) + \e \F_1(u_0) + \e^2 \F_2(u_0)
\]
and from the minimality of $u_{\e}$ we deduce
\[
\F_{\e}(u_{\e}) \le \F_0(u_0) + \e \F_1(u_0) + \e^2 \F_2(u_0).
\]
Writing $r_{\e} \coloneqq \frac{u_{\e} - u_0}{\e}$, expanding, and rearranging the terms in the previous inequality we arrive at 
\begin{equation}
\label{estimateforre}
\frac{1}{2}\int_{\Omega}|\nabla r_{\e}|^2\,d\bm{x} + \frac{1}{2\e}\int_{\partial \Omega}\left(r_{\e} + \e \partial_{\nu}u_0\right)^2\,d\mathcal{H}^{N-1} \le \frac{\e^2}{2}\int_{\Omega}|\nabla u_1|^2\,d\bm{x}.
\end{equation}
Since $\partial \Omega$ is of class $C^{1,1}$, $f \in L^2(\Omega)$, and $g \in H^{3/2}(\partial \Omega)$, by standard elliptic estimates (see Theorem 2.4.2.5 in \cite{grisvard85}) the solution $u_0 \in H^1(\Omega)$ to the Dirichlet problem (\ref{R0}) belongs to $H^2(\Omega)$ with
\[
\|u_0\|_{H^2(\Omega)} \le k_1\left(\|f\|_{L^2(\Omega)} + \|g\|_{H^{3/2}(\Omega)}\right).
\]
In turn, by standard trace theorems (see Theorem 18.40 in \cite{leoni}), we have that $\partial_{\nu}u_0 \in H^{1/2}(\partial \Omega)$, and so there is $z_0 \in H^1(\Omega)$ such that $z_0 = -\partial_{\nu}u_0$ on $\partial \Omega$ in the sense of traces and
\[
\|z_0\|_{H^1(\Omega)} \le k_2 \|\partial_{\nu}u_0\|_{H^{1/2}(\partial \Omega)} \le k_3\|u_0\|_{H^2(\Omega)} \le c\left(\|f\|_{L^2(\Omega)} + \|g\|_{H^{3/2}(\Omega)}\right).
\]
Since $u_1 \in H^1(\Omega)$ is a minimizer of
\[
v \mapsto \int_{\Omega}|\nabla v|^2\,d\bm{x}
\]
over all functions $v$ with $v = - \partial_{\nu}u_0$ on $\partial \Omega$, we have that 
\[
\|\nabla u_1\|_{L^2(\Omega;\RR^N)} \le \|\nabla z_0\|_{L^2(\Omega;\RR^N)} \le c\left(\|f\|_{L^2(\Omega)} + \|g\|_{H^{3/2}(\Omega)}\right).
\]
The previous estimate, together with (\ref{estimateforre}), gives the desired result.
\end{proof}

\subsection{The non-mixed problem: Gamma-convergences of all orders}
In this section we prove asymptotic expansions by $\Gamma$-convergence of any order for $\F_{\e}$ and derive asymptotic expansions for $u_{\e}$, i.e. the solution to $(\ref{Re})$.
\begin{thm}
\label{kgcJ}
Given $k \in \NN$, let $j \in \NN$ be such that $k = 2j - 1$ or $k = 2j$, $\Omega$ be as in $($\ref{Hj}$)$, $f \in L^2(\Omega)$, $g \in H^{1/2 + j}(\partial \Omega)$, and for every $m \in \{1, \dots, j\}$ let $u_m \in H^1(\Omega)$ be the solution to the Dirichlet problem
\begin{equation}
\label{Rm}
\left\{
\arraycolsep=1.4pt\def\arraystretch{1.6}
\begin{array}{rll}
\Delta u_m = & 0 &  \text{ in } \Omega, \\
u_m = & -\partial_{\nu}u_{m-1} & \text{ on } \partial \Omega,
\end{array}
\right.
\end{equation}
where $u_0$ is the solution to $(\ref{R0})$. Let $\F_{\e}^{(k+1)}$ be defined recursively by
\[
\F_{\e}^{(k+1)} \coloneqq \frac{\F_{\e}^{(k)} - \F_{k}(u_0)}{\e},
\]
where $\F_{\e}^{(1)}$ is given as in $(\ref{Fe1J})$ and the functionals $\F_i$, for $i \in \{1,\dots, k+1\}$, are given by
\[
\F_{2m + 1}(v) \coloneqq
\left\{
\arraycolsep=1.4pt\def\arraystretch{1.9}
\begin{array}{ll}
\displaystyle -\frac{1}{2}\int_{\partial \Omega}\left(\partial_{\nu}u_m\right)^2\,d\mathcal{H}^{N-1} & \text{ if } v = u_0, \\
+\infty & \text{ otherwise},
\end{array} 
\right. 
\]
and
\[
\F_{2m }(v) \coloneqq
\left\{
\arraycolsep=1.4pt\def\arraystretch{1.9}
\begin{array}{ll}
\displaystyle +\frac{1}{2}\int_{\Omega}|\nabla u_m|^2\,d\bm{x} & \text{ if } v = u_0, \\
+\infty & \text{ otherwise}.
\end{array} 
\right. 
\]
Then the family $\{\F_{\e}^{(i)}\}_{\e}$ $\Gamma$-converges in $L^2(\Omega)$ to the functional $\F_i$ for every $i \in \{2,\dots, k+1\}$. In particular, if $u_{\e} \in H^1(\Omega)$ is the solution to $(\ref{Re})$, then 
\[
\F_{\e}(u_{\e}) = \sum_{i = 0}^{k+1}\e^i\F_i(u_0) + o\left(\e^{k+1}\right).
\]
\end{thm}

\begin{proof}
Notice that for $k = 1$ we have that $j = 1$ and so the statement reduces to the one of \Cref{2gcJ}. The result for $ k \ge 2$ follows by induction from arguments similar to the ones of \Cref{1gcJ} and \Cref{2gcJ} (depending on the parity of $k$). We omit the details. 
\end{proof}

\begin{cor}
\label{estimatesrkJ}
Under the assumptions of Theorem \ref{kgcJ}, and for an odd value of $k \in \NN$, let $u_{\e}$, $u_0$, $u_i$ be solutions to $(\ref{Re})$, $(\ref{R0})$, and $(\ref{Rm})$, respectively. Then there exists a constant $c > 0$, independent of $\e$, such that for every $j \in\{1, \dots, (k+1)/2\}$
\[
\label{estimatereps}
\begin{aligned}
\left\|u_{\e} - \sum_{i = 0}^{j-1} \e^i u_i\right\|_{H^1(\Omega)} \le &\ C\e^{j}\left(\|f\|_{L^2(\Omega)} + \|g\|_{H^{1/2 + j}(\Omega)}\right), \\
\left\|u_{\e} - \sum_{i = 0}^{j-1} \e^i u_i + \e \partial_{\nu}u_j\right\|_{L^2(\partial \Omega)} \le &\ C\e^{1/2 + j}\left(\|f\|_{L^2(\Omega)} + \|g\|_{H^{1/2 + j}(\Omega)}\right).
\end{aligned}
\]
\end{cor}

\begin{proof}
The proof is analogous to the one of \Cref{estimatesr1J} and therefore we omit the details.
\end{proof}
\section{The case of mixed boundary conditions}
In this section we prove our main results regarding the higher order $\Gamma$-limits for the functional $\F_{\e}$ defined as in (\ref{Fe}). 

\subsection{Preliminary results}
Throughout the section $\Omega$ is assumed to be as in the statement of \Cref{H2dec}. We recall that we use the following notations: given a function $v = v(\bm{x})$ where $\bm{x} = (x,y)$, we denote by $\bar{v}$ the function 
\begin{equation}
\label{origin}
\bar{v}(r,\theta) \coloneqq v(r \cos \theta, r \sin \theta),
\end{equation}
and with a slight abuse of notation we write $v = \bar{v}(r,\theta)$. Moreover, we denote by $\bar{v}^{(i)}$ the function
\begin{equation}
\label{polarxi}
\bar{v}^{(i)}(r_i,\theta_i) \coloneqq v(\bm{x}_i + r_i(\cos \theta_i, \sin \theta_i)),
\end{equation}
where the polar coordinates $(r_i,\theta_i)$ are as in \Cref{H2dec}. 
Furthermore, recall that $\bar{\varphi} \in C^{\infty}([0,\infty))$ is such that $\bar{\varphi} \equiv 1$ in $[0,\rho/2]$ and $\bar{\varphi} \equiv 0$ outside $[0,\rho]$. 

\begin{prop}
\label{IBP}
Let $N = 2$, $\Omega$ be as in \Cref{H2dec}, $f \in L^2(\Omega)$, $g \in H^{3/2}(\partial \Omega)$, and let $u_0 \in H^1(\Omega)$ be the solution to $(\ref{P0})$. Then  
\[
\int_{\Omega}\left(\nabla u_0 \cdot \nabla \psi + f\psi \right) \,d\bm{x} = \int_{\Gamma_D}\partial_{\nu}u^0_{\reg}\psi \,d\mathcal{H}^1 - \sum_{i = 1}^2 \frac{c_i}{2}\int_0^{\rho}\bar{\varphi}(r_i)r_i^{-1/2}\bar{\psi}^{(i)}(r_i,0)\,dr_i
\]
for every $\psi \in H^1(\Omega)$, where $u^0_{\reg},c_i$ and $\bar{\varphi}$ are given as in \Cref{H2dec}.
\end{prop}

\begin{proof}
By \Cref{H2dec}, given $\psi \in H^1(\Omega)$, we get
\begin{equation}
\label{gradgrad}
\int_{\Omega}\nabla u_0 \cdot \nabla \psi \,d\bm{x} = \int_{\Omega}\nabla u^0_{\reg} \cdot \nabla \psi \,d\bm{x} + \sum_{i = 1}^2c_i\int_0^{\pi}\int_0^{\rho}\left(\partial_{r_i}\bar{S}_i \partial_{r_i} \bar{\psi}^{(i)} + r_i^{-2}\partial_{\theta_i}\bar{S}_i \partial_{\theta_i} \bar{\psi}^{(i)} \right)r_i\,dr_id\theta_i.
\end{equation}
Since the function $u^0_{\reg}$ belongs to $H^2(\Omega)$ and satisfies a homogenous Neumann boundary condition on $\Gamma_N$, the divergence theorem yields
\begin{equation}
\int_{\Omega}\nabla u^0_{\reg} \cdot \nabla \psi \,d\bm{x} = \int_{\Omega}-\Delta u^0_{\reg}\psi\,d\bm{x} + \int_{\Gamma_D}\partial_{\nu}u^0_{\reg}\psi\,d\mathcal{H}^1. \label{uregneumann}
\end{equation}
To rewrite the second term on the right-hand side of (\ref{gradgrad}), we consider the auxiliary function 
\[
\bar{\Phi}(r_i, \theta_i) \coloneqq r_i\partial_{r_i}\bar{S}_i(r_i,\theta_i)\bar{\psi}^{(i)}(r_i,\theta_i);
\] 
indeed, a simple computation shows that $\bar{\Phi} \in W^{1,1}((0,\rho) \times (0,\pi))$ and thus $\bar{\Phi}(\cdot, \theta_i)$ is absolutely continuous for $\mathcal{L}^1$-a.e.\@ $\theta_i \in (0,\pi)$. For any such $\theta_i$, by the fundamental theorem of calculus, we have that
\begin{align}
0 = \bar{\Phi}(\rho,\theta_i) - \bar{\Phi}(0,\theta_i) = \int_0^{\rho}\partial_{r_i}\bar{\Phi}(r_i,\theta_i)\,dr_i = \int_0^{\rho} \left(\partial_{r_i}\bar{S}_i \bar{\psi}^{(i)} + r_i \partial_{r_i}^2\bar{S}_i \bar{\psi}^{(i)} + r_i \partial_{r_i}\bar{S}_i \partial_{r_i} \bar{\psi}^{(i)}\right)\,dr_i. \label{ftcPHI}
\end{align}
Similarly, noticing that the function $\bar{\Psi}(r_i,\theta_i) \coloneqq r_i^{-1}\partial_{\theta_i}\bar{S}_i(r_i,\theta_i) \bar{\psi}^{(i)}(r_i,\theta_i)$ belongs to the space $W^{1,1}((0,\rho) \times (0,\pi))$, and reasoning as above we find that  
\begin{align}
-\frac{1}{2}\bar{\varphi}(r_i)r_i^{-1/2}\bar{\psi}^{(i)}(r_i,0) = \bar{\Psi}(r_i,\pi) - \bar{\Psi}(r_i,0) = &\ \int_0^{\pi}\partial_{\theta_i}\bar{\Psi}(r_i, \theta_i)\,d\theta_i \notag \\
= &\ \int_0^{\pi}r_i^{-1}\left(\partial_{\theta_i}^2\bar{S}_i \bar{\psi}^{(i)} + \partial_{\theta_i}\bar{S}_i \partial_{\theta_i} \bar{\psi}^{(i)} \right)\,d\theta_i \label{ftcPSI}
\end{align}
holds for $\mathcal{L}^1$-a.e. $r_i \in (0,\rho)$. Combining the identities (\ref{ftcPHI}) and (\ref{ftcPSI}), we get
\begin{align*}
\int_0^{\pi}\int_0^{\rho} & \left(\partial_{r_i}\bar{S}_i \partial_{r_i} \bar{\psi}^{(i)} + r_i^{-2}\partial_{\theta_i}\bar{S}_i \partial_{\theta_i} \bar{\psi}^{(i)} \right)r_i\,dr_id\theta_i \\
= &\ - \frac{1}{2}\int_0^{\rho}\bar{\varphi}(r_i)r_i^{-1/2} \bar{\psi}^{(i)}(r_i,0)\,dr_i  - \int_0^{\pi}\int_0^{\rho} \bar{\psi}^{(i)} \left(\partial_{r_i}^2\bar{S}_i + r_i^{-1}\partial_{r_i}\bar{S}_i + r_i^{-2}\partial_{\theta_i}^2\bar{S}_i \right)r_i\,dr_id\theta_i \notag \\
= &\ - \frac{1}{2} \int_0^{\rho}\bar{\varphi}(r_i)r_i^{-1/2} \bar{\psi}^{(i)}(r_i,0)\,dr_i - \int_0^{\pi}\int_0^{\rho} \bar{\psi}^{(i)} \Delta_{(r_i,\theta_i)} \bar{S}_i r_i \,dr_id\theta_i.
\end{align*}
Consequently, the desired formula follows from the previous equality, (\ref{gradgrad}), (\ref{uregneumann}), and upon noticing that 
\[
\int_{\Omega}f\psi \,d\bm{x} = \int_{\Omega}\Delta u^0_{\reg}\psi\,d\bm{x} + \sum_{i = 1}^2 c_i\int_0^{\pi}\int_0^{\rho}\bar{\psi}^{(i)} \Delta_{(r_i,\theta_i)} \bar{S}_i r_i \,dr_id\theta_i.
\] 
This concludes the proof.
\end{proof}

In the following theorem we present an estimate that will prove instrumental for the proofs of our compactness results, namely \Cref{1gc-cpt} and \Cref{2gc-cpt}.

\begin{thm}
\label{lemmacpt}
There exists a constant $\kappa$ such that for any $R > 0$ and $h \in H^1(B_R^+(\bm{0}))$,
\[
\int_0^R x^{-1/2}|h(x,0)|\,dx \le \kappa \left(R\int_{B_R^+(\bm{0})}|\nabla h(\bm{x})|^2\,d\bm{x}\right)^{1/2} + \kappa \left(\int_0^R h(x,0)^2\,dx\right)^{1/2},
\]
where $h(\cdot,0)$ indicates the trace of $h$ on the positive real axis.
\end{thm}

We begin by adapting \Cref{hardyonballs} to our framework.
\begin{lem}
\label{hardy}
There exists a constant $\overline{\kappa}$ such that for any $R > 0$ and $h \in H^1(B_R^+(\bm{0}))$,
\[
\int_{B_R^+(\bm{0})}\frac{h(\bm{x})^2}{|\bm{x}|^2\left(1 + \log R - \log|\bm{x}|\right)^2}\,d\bm{x} \le \overline{\kappa} \left(\int_{B_R^+(\bm{0})}\left|\nabla h(\bm{x})\right|^2\,d\bm{x} + \frac{1}{R}\int_0^R h(x,0)^2\,dx \right),
\]
where $h(\cdot,0)$ indicates the trace of $h$ on the positive real axis.
\end{lem}
\begin{proof}
Since $B_R^+(\bm{0})$ is an extension domain, we can find $\hat{h} \in H^1(B_R(\bm{0}))$ such that $\hat{h}(\bm{x}) = h(\bm{x})$ for $\mathcal{L}^2$-a.e. $\bm{x} \in B_R^+(\bm{0})$ and with the property that
\[
\begin{aligned}
\|\hat{h}\|_{L^2(B_R(\bm{0}))} \le &\ C_1\|h\|_{L^2(B_R^+(\bm{0}))}, \\
\|\nabla \hat{h}\|_{L^2(B_R(\bm{0});\RR^2)} \le &\ C_1 \|\nabla h\|_{L^2(B_R^+(\bm{0});\RR^2)}, 
\end{aligned}
\]
for some constant $C_1 > 0$ independent of $R$. \Cref{hardyonballs} applied to the function $\hat{h}$ and the previous estimates yield 
\[
\int_{B_R^+(\bm{0})}\frac{h(\bm{x})^2}{|\bm{x}|^2\left(1 + \log R - \log|\bm{x}|\right)^2}\,d\bm{x} \le C_2\left(\int_{B_R^+(\bm{0})}\left|\nabla h(\bm{x})\right|^2\,d\bm{x} + \frac{1}{R^2}\int_{B_R^+(\bm{0})}h(\bm{x})^2\,d\bm{x}\right),
\]
for some constant $C_2 > 0$ independent of $h$ and $R$. By \Cref{equivnorm}, together with a simple rescaling argument, we deduce that
\[
\frac{1}{R^2} \int_{B_R^+(\bm{0})}h(\bm{x})^2\,d\bm{x} \le C_3\left(\int_{B_R^+(\bm{0})}|\nabla h(\bm{x})|^2\,d\bm{x} + \frac{1}{R}\int_0^R h(x,0)^2\,dx\right)
\]
for some constant $C_3 > 0$, which is again independent of both $h$ and $R$; this concludes the proof.
\end{proof}

\begin{proof}[Proof of \Cref{lemmacpt}]
By the fundamental theorem of calculus,
\[
\bar{h}(r,\theta) = \bar{h}(r,0) + \int_0^{\theta}\partial_{\theta}\bar{h}(r,\alpha)\,d\alpha,
\]
and so, multiplying both sides by $r^{-1/2}$ and integrating over $B_R^+(\bm{0})$, we get
\begin{align*}
-\int_0^Rr^{-1/2}\bar{h}(r,0)\,dr = &\ -\frac{1}{\pi}\int_0^{\pi}\int_0^Rr^{-1/2}\bar{h}(r,\theta)\,dr d\theta + \frac{1}{\pi}\int_0^{\pi}\int_0^R\int_0^{\theta}r^{-1/2}\partial_{\theta}\bar{h}(r,\alpha)\,d\alpha dr d\theta \\
= &\ -\frac{1}{\pi}\int_0^{\pi}\int_0^Rr^{-1/2}\bar{h}(r,\theta)\,dr d\theta + \int_0^{\pi}\int_0^R\frac{(\pi - \theta)}{\pi}r^{-1/2}\partial_{\theta}\bar{h}(r,\theta)\,drd\theta, 
\end{align*}
where the last equality follows from Fubini's theorem. In particular,
\begin{equation}
\label{lemmacpt1}
\int_0^Rr^{-1/2}|\bar{h}(r,0)|\,dr \le \frac{1}{\pi}\int_0^{\pi}\int_0^Rr^{-1/2}|\bar{h}(r,\theta)|\,dr d\theta + \int_0^{\pi}\int_0^R r^{-1/2}|\partial_{\theta}\bar{h}(r,\theta)|\,drd\theta,
\end{equation}
and thus we proceed to estimate the terms on the right-hand side of (\ref{lemmacpt1}). Passing to cartesian coordinates,
\begin{align*}
\int_0^{\pi}\int_0^Rr^{-1/2}|\bar{h}(r,\theta)|\,dr d\theta = &\ \int_{B_R^+(\bm{0})}\frac{|h(\bm{x})|}{|\bm{x}|\left(1 + \log R - \log |\bm{x}|\right)}\frac{\left(1 + \log R - \log |\bm{x}|\right)}{|\bm{x}|^{1/2}}\,d\bm{x} \\
\le &\ (5 \pi R)^{1/2}\left(\int_{B_R^+(\bm{0})}\frac{h(\bm{x})^2}{|\bm{x}|^2\left(1 + \log R - \log |\bm{x}|\right)^2}\,d\bm{x}\right)^{1/2},
\end{align*}
where in the last step we have used H\"older's inequality together with the fact that
\[
\int_{B_R^+(\bm{0})}\frac{\left(1 + \log R - \log |\bm{x}|\right)^2}{|\bm{x}|}d\bm{x} =  \pi \int_0^R\left(1 + \log R - \log r\right)^2\,dr = 5 \pi R.
\]
Then, from \Cref{hardy} we deduce that 
\begin{equation}
\label{lemmacpt2}
\int_0^{\pi}\int_0^Rr^{-1/2}|\bar{h}(r,\theta)|\,dr d\theta \le (5 \pi \overline{\kappa})^{1/2}\left(R \int_{B_R^+(\bm{0})} |\nabla h(\bm{x})|^2\,d\bm{x} + \int_0^Rh(x,0)^2\,dx \right)^{1/2}.
\end{equation}
On the other hand, H\"older's inequality yields
\begin{align}
\label{lemmacpt3}
\int_0^{\pi}\int_0^R r^{-1/2}|\partial_{\theta}\bar{h}(r,\theta)|\,drd\theta \le &\ \left(\pi R \int_0^{\pi}\int_0^Rr^{-1}|\partial_{\theta}\bar{h}(r,\theta)|^2\,drd\theta\right)^{1/2} \notag \\
\le &\ \left(\pi R\int_{B_R^+(\bm{0})} |\nabla h(\bm{x})|^2\,d\bm{x}\right)^{1/2},
\end{align}
and so the desired inequality follows from (\ref{lemmacpt1}), (\ref{lemmacpt2}), and (\ref{lemmacpt3}).
\end{proof}

\subsection{Mixed boundary conditions: Gamma-convergence of order one}
In this section we prove \Cref{1gc-cpt} and \Cref{1gc}. We recall that we use the notations (\ref{origin}) and (\ref{polarxi}).
\begin{proof}[Proof of \Cref{1gc-cpt}]
By \Cref{convofmin} we have that $v_n \to u_0$ in $H^1(\Omega)$. For every $n \in \NN$, let $z_n \in L^2(\Omega)$ be such that $v_n = u_0 + \e_n\sqrt{|\log \e_n|}z_n$. Then $\F_{\e_n}^{(1)}(v_n)$ can be rewritten as  
\[
\F_{\e}^{(1)}(v_n) = \frac{1}{\sqrt{|\log\e_n|}}\int_{\Omega} \left( \nabla u_0 \cdot \nabla z_n + f z_n \right)\,d\bm{x} + \frac{\e_n}{2}\int_{\Omega}|\nabla z_n|^2\,d\bm{x} + \frac{1}{2}\int_{\Gamma_D}z_n^2\,d\mathcal{H}^1,
\]
and an application of \Cref{IBP} yields
\begin{equation}
\label{Fe1cpt}
\begin{aligned}
\F_{\e_n}^{(1)}(v_n) = &\ \frac{1}{\sqrt{|\log\e_n|}} \left( \int_{\Gamma_D}\partial_{\nu}u^0_{\reg}z_n\,d\mathcal{H}^1 - \sum_{i = 1}^2\frac{c_i}{2}\int_0^{\rho}\bar{\varphi}(r_i)r_i^{-1/2}\bar{z}^{(i)}_n(r_i,0)\,dr_i\right) \\
&\ \quad + \frac{\e_n}{2}\int_{\Omega}|\nabla z_n|^2\,d\bm{x} + \frac{1}{2} \int_{\Gamma_D}z_n^2\,d\mathcal{H}^1.
\end{aligned}
\end{equation}
For $n$ large enough so that $2\e_n \le \rho$, we write
\begin{equation}
\label{split}
\int_0^{\rho}\bar{\varphi}(r_i)r_i^{-1/2}\bar{z}^{(i)}_n(r_i,0)\,dr_i = \int_0^{\e_n}r_i^{-1/2}\bar{z}^{(i)}_n(r_i,0)\,dr_i + \int_{\e_n}^{\rho}\bar{\varphi}(r_i)r_i^{-1/2}\bar{z}^{(i)}_n(r_i,0)\,dr_i
\end{equation}
and proceed to estimate both terms on the right-hand side separately. By \Cref{lemmacpt} we obtain
\begin{equation}
\label{RN1cpt}
\int_0^{\e_n}r_i^{-1/2}|\bar{z}^{(i)}_n(r_i,0)|\,dr_i \le \kappa \left(\e_n\int_{B_{\e_n}(\bm{x}_i) \cap \Omega}|\nabla z_n|^2\,d\bm{x}\right)^{1/2} + \kappa \left(\int_0^{\e_n}\bar{z}^{(i)}_n(r_i,0)^2\,dr_i\right)^{1/2},
\end{equation}
while by H\"older's inequality we get
\begin{equation}
\label{RN2cpt}
\int_{\e_n}^{\rho}\bar{\varphi}(r_i)r_i^{-1/2}|\bar{z}^{(i)}_n(r_i,0)|\,dr_i \le \sqrt{\log \rho + |\log \e_n|}\left(\int_{\e_n}^{\rho}\bar{z}^{(i)}_n(r_i,0)^2\,dr_i\right)^{1/2}.
\end{equation}
Consequently, from (\ref{Fe1cpt}), (\ref{RN1cpt}), and (\ref{RN2cpt}) we deduce that
\[
\begin{aligned}
\F_{\e_n}^{(1)}(v_n) \ge &\ \frac{1}{2}\|z_n\|_{L^2(\Gamma_D)}^2 - \left(\frac{\|\partial_{\nu}u^0_{\reg}\|_{L^2(\Gamma_D)}}{\sqrt{|\log\e_n|}} + \frac{|c_i|(\kappa + \sqrt{\log \rho + |\log\e_n|})}{2\sqrt{|\log\e_n|}}\right)\|z_n\|_{L^2(\Gamma_D)} \\
&\ \quad + \frac{1}{2}\|\e_n^{1/2}\nabla z_n\|_{L^2(\Omega;\RR^2)}^2 - \frac{|c_i|\kappa}{2\sqrt{|\log\e_n|}}\|\e_n^{1/2}\nabla z_n\|_{L^2(\Omega;\RR^2)},
\end{aligned}
\]
and so (\ref{1cpt1}) and (\ref{1cpt2}) are proved at once.
\end{proof}

\begin{proof}[Proof of \Cref{1gc}]
\textbf{Step 1: }Let $\e_n \to 0^+$ and $\{v_n\}_n$ be a sequence of functions in $L^2(\Omega)$ such that $v_n \to v$ in $L^2(\Omega)$. Reasoning as in the proof of \Cref{0gc}, we can assume without loss of generality that 
\[
\liminf_{n \to \infty} \F_{\e_n}^{(1)}(v_n) = \lim_{n \to \infty} \F_{\e_n}^{(1)}(v_n) < \infty.
\]
In particular, $\F_{\e_n}^{(1)}(v_n) < \infty$ for every $n$ sufficiently large. Let $\{v_{n_k}\}_k$ be a subsequence of $\{v_n\}_n$ given as in \Cref{1gc-cpt} and define
\begin{equation}
\label{xiin}
\bar{\xi}_n^{(i)}(r_i) \coloneqq \frac{c_i}{2\sqrt{|\log\e_n|}}\bar{\varphi}(r_i)r_i^{-1/2}.
\end{equation}
Arguing as in the proof of \Cref{1gc-cpt} (see (\ref{Fe1cpt}) and (\ref{RN1cpt})) we arrive at
\begin{align}
\F_{\e_{n_k}}^{(1)}(v_{n_k}) \ge &\ \frac{1}{2}\|z_{n_k}\|_{L^2(\Gamma_D)}^2 - \left(\frac{\|\partial_{\nu}u^0_{\reg}\|_{L^2(\Gamma_D)}}{\sqrt{|\log\e_{n_k}|}} + \frac{|c_i|\kappa}{2\sqrt{|\log\e_{n_k}|}}\right)\|z_{n_k}\|_{L^2(\Gamma_D)} \notag \\
&\ \quad - \frac{|c_i|\kappa}{2\sqrt{|\log\e_{n_k}|}}\|\e_{n_k}^{1/2}\nabla z_{n_k}\|_{L^2(\Omega;\RR^2)} + \frac{1}{2}\|\e_{n_k}^{1/2}\nabla z_{n_k}\|_{L^2(\Omega;\RR^2)}^2 \notag \\
&\ \quad - \sum_{i = 1}^2\int_{\e_{n_k}}^{\rho}\bar{\xi}_{n_k}^{(i)}(r_i)\bar{z}^{(i)}_{n_k}(r_i,0)\,dr_i. \label{Fenk1}
\end{align}
Then, as $k \to \infty$, we have
\begin{align}
\liminf_{k \to \infty}\F_{\e_{n_k}}^{(1)}(v_{n_k}) \ge &\ \liminf_{k \to \infty}\sum_{i = 1}^2\int_{\e_{n_k}}^{\rho}\left(\frac{1}{2}\bar{z}^{(i)}_{n_k}(r_i,0)^2 - \bar{\xi}^{(i)}_{n_k}(r_i)\bar{z}^{(i)}_{n_k}(r_i,0)\right)\,dr_i \notag \\
= &\ \liminf_{k \to \infty}\sum_{i = 1}^2\left[\frac{1}{2}\int_{\e_{n_k}}^{\rho}\left(\bar{z}^{(i)}_{n_k}(r_i,0) - \bar{\xi}_{n_k}^{(i)}(r_i)\right)^2\,dr_i - \frac{1}{2}\int_{\e_{n_k}}^{\rho}\bar{\xi}_{n_k}^{(i)}(r_i)^2\,dr_i\right] \notag \\
\ge &\  -\frac{1}{2}\sum_{i = 1}^2\liminf_{k \to \infty}\int_{\e_{n_k}}^{\rho}\bar{\xi}_{n_k}^{(i)}(r_i)^2\,dr_i \notag \\
= &\ -\frac{1}{8}\sum_{i = 1}^2c_i^2\liminf_{k \to \infty}\frac{1}{|\log\e_{n_k}|}\int_{\e_{n_k}}^{\rho}\bar{\varphi}(r_i)^2r_i^{-1}\,dr_i \notag \\
\ge &\ -\frac{1}{8}\sum_{i = 1}^2c_i^2\liminf_{k \to \infty}\frac{1}{|\log\e_{n_k}|}\left(\log \rho + |\log \e_{n_k}|\right) = -\frac{1}{8}\sum_{i = 1}^2c_i^2, \label{gammalininf1}
\end{align}
where in the second to last step we have used (\ref{xiin}).
\newline
\textbf{Step 2: }For every $v \in L^2(\Omega) \setminus \{u_0\}$, the constant sequence $v_n = v$ is a recovery sequence. Then let $v = u_0$ and consider the radial function $\zeta_{i,n}$ given in polar coordinates at $\bm{x}_i$ by
\begin{equation}
\label{zetain}
\bar{\zeta}_{i,n}(r_i) \coloneqq \bar{\xi}_n^{(i)}(r_i)\left(1 - \bar{\varphi}\left(\frac{\rho}{\e_n}r_i\right)\right) = \frac{c_i}{2\sqrt{|\log \e_n|}}\bar{\varphi}(r_i)\left(1 - \bar{\varphi}\left(\frac{\rho}{\e_n}r_i\right)\right)r_i^{-1/2},
\end{equation}
where $\bar{\xi}^{(i)}_n$ is the function defined in (\ref{xiin}). We define 
\begin{equation}
\label{recFe1}
z_n(\bm{x}) \coloneqq
\left\{
\arraycolsep=1.4pt\def\arraystretch{1.6}
\begin{array}{ll}
\zeta_{i,n}(\bm{x}) & \text{ if } \bm{x} \in B_r(\bm{x}_i) \cap \Omega \text{ with } r < \rho, \\
0 & \text{ otherwise.}
\end{array}
\right.
\end{equation}
Notice that if we let
\[
\bar{\Psi}_{i,n}(r_i) \coloneqq \bar{\varphi}(r_i)\left(1 - \bar{\varphi}\left(\frac{\rho}{\e_n}r_i\right)\right),
\]
then $\bar{\Psi}_{i,n} \colon \RR^+ \to [0,1]$ and satisfies
\begin{equation}
\label{propertiesPsi}
\left\{
\arraycolsep=1.4pt\def\arraystretch{1.6}
\begin{array}{rll}
\bar{\Psi}_{i,n}(r_i) = & 1  & \text{ if } \e_n \le r_i < \rho/2, \\
\bar{\Psi}_{i,n}(r_i) = & 0 & \text{ if } 0 \le r_i \le \e_n/2 \text{ or } \rho \le r, \\
|\bar{\Psi}_{i,n}'(r_i)| \le & \frac{c}{\e_n} & \text{ if }  \e_n/2 \le r_i \le \e_n, \\
|\bar{\Psi}_{i,n}'(r_i)| \le & c & \text{ if } \rho/2 \le r_i \le \rho,
\end{array}
\right.
\end{equation}
for some constant $c > 0$ independent of $n$. Finally, set 
\[
v_n \coloneqq u_0 + \e_n\sqrt{|\log \e_n|}z_n.
\]
Notice that $v_n \to u_0$ in $L^2(\Omega)$ since the sequence $\{z_n\}_n$ is uniformly bounded in $L^2(\Omega)$, indeed
\begin{align*}
\int_{\Omega}z_n^2\,d\bm{x} \le \sum_{i = 1}^2 \frac{c_i^2\pi}{4|\log \e_n|}\int_{\e_n/2}^{\rho}r_i^{-1}\,dr_i = \frac{\pi (\log \rho + |\log \e_n| + \log 2)}{4|\log \e_n|}\sum_{i = 1}^2c_i^2.
\end{align*} 
Next, we claim that $\e_n^{1/2}\nabla z_n \to \bm{0}$ in $L^2(\Omega; \RR^2)$. Indeed, using the notation above we have that 
\[
\bar{\zeta}_{i,n}(r_i) = \frac{c_i}{2\sqrt{|\log \e_n|}}\bar{\Psi}_{i,n}(r_i)r_i^{-1/2},
\]
and therefore
\begin{align}
\label{egrad0}
\e_n\int_{\Omega}|\nabla z_n|^2\,d\bm{x} = &\ \frac{\e_n}{|\log \e_n|}\left(\sum_{i = 1}^2\frac{c_i^2\pi}{4}\right)\int_0^{\rho}\left(\bar{\Psi}_{i,n}'(r_i)r_i^{-1/2} -\frac{1}{2}r_i^{-3/2}\bar{\Psi}_{i,n}(r_i)\right)^2r_i\,dr_i \notag \\
\le &\ \frac{\e_n}{|\log \e_n|} \left(\sum_{i = 1}^2\frac{c_i^2\pi}{2}\right)\int_0^{\rho}\left(\bar{\Psi}_{i,n}'(r_i)^2 + \frac{1}{4}r_i^{-2}\bar{\Psi}_{i,n}(r_i)^2\right)\,dr_i.
\end{align}
From (\ref{propertiesPsi}) we see that
\begin{equation}
\label{psi'2}
\int_0^{\rho}\bar{\Psi}_{i,n}'(r_i)^2\,dr_i = \int_{\e_n/2}^{\e_n}\bar{\Psi}_{i,n}'(r_i)^2\,dr_i + \int_{\rho/2}^{\rho}\bar{\Psi}_{i,n}'(r_i)^2\,dr_i \le c^2\left(\frac{1}{2\e_n} + \frac{\rho}{2}\right)
\end{equation}
and
\begin{equation}
\label{r^-2}
\int_0^{\rho}r_i^{-2}\bar{\Psi}_{i,n}(r_i)^2\,dr_i \le \int_{\e_n/2}^{\rho}r_i^{-2}\,dr_i = \frac{2}{\e_n} - \frac{1}{\rho}.
\end{equation}
Combining (\ref{egrad0}) with the estimates (\ref{psi'2}) and (\ref{r^-2}) we obtain
\begin{equation}
\label{egradto0}
\e_n\int_{\Omega}|\nabla z_n|^2\,d\bm{x} \le \frac{\e_n}{|\log \e_n|} \left(\sum_{i = 1}^2\frac{c_i^2\pi}{2}\right)\left(\frac{c^2}{2\e_n} + \frac{c^2\rho}{2} + \frac{1}{2\e_n} - \frac{1}{4\rho}\right) \to 0
\end{equation}
and the claim is proved. From (\ref{Fe1cpt}), using (\ref{split}), (\ref{RN1cpt}), (\ref{RN2cpt}), and (\ref{xiin}) we have
\begin{align}
\F_{\e_n}^{(1)}(v_n) \le &\ \frac{1}{2}\|z_n\|_{L^2(\Gamma_D)}^2 + \left(\frac{\|\partial_{\nu}u^0_{\reg}\|_{L^2(\Gamma_D)}}{\sqrt{|\log\e_n|}} + \frac{|c_i|\kappa}{2\sqrt{|\log\e_n|}}\right)\|z_n\|_{L^2(\Gamma_D)} + \frac{1}{2}\|\e_n^{1/2}\nabla z_n\|_{L^2(\Omega; \RR^2)}^2 \notag \\
&\ \quad + \frac{|c_i|\kappa}{2\sqrt{|\log\e_n|}}\|\e_n^{1/2}\nabla z_n\|_{L^2(\Omega;\RR^2)} - \sum_{i = 1}^2\int_{\e_n}^{\rho}\bar{\xi}_n^{(i)}(r_i)\bar{\zeta}_{i,n}(r_i)\,dr_i, \label{Fen1vnsup}
\end{align}
By (\ref{egradto0}) we have that the second, third, and fourth member on the right-hand side of the previous inequality vanish as $n \to \infty$. Since $\bar{\varphi}\left(\frac{\rho}{\e_n}r_i\right) = 0$ for $r_i \in [\e_n,\rho]$, by (\ref{xiin}) and (\ref{zetain}), 
\begin{equation}
\label{z=x}
\bar{\zeta}_{i,n} = \bar{\xi}^{(i)}_n \quad \text{ in } [\e_n,\rho].
\end{equation}
Consequently, from (\ref{xiin}), (\ref{z=x}), (\ref{recFe1}), and the fact that $\bar{\varphi} \equiv 1$ in $[0,\rho/2]$,
\begin{align}
\limsup_{n \to \infty}\F_{\e_n}^{(1)}(v_n) \le &\ \limsup_{n \to \infty}\left\{\frac{1}{2}\|z_n\|_{L^2(\Gamma_D)}^2 - \sum_{i = 1}^2\int_{\e_n}^{\rho}\bar{\xi}_n^{(i)}(r_i)\bar{\zeta}_{i,n}(r_i)\,dr_i\right\} \notag \\
= &\ \limsup_{n \to \infty}\sum_{i = 1}^2\int_{\e_n}^{\rho}\left(\frac{1}{2}\bar{\zeta}_{i,n}(r_i)^2 - \bar{\xi}_n^{(i)}(r_i)\bar{\zeta}_{i,n}(r_i)\right)\, dr_i \notag \\ 
= &\ \limsup_{n \to \infty}\sum_{i = 1}^2-\frac{1}{2} \int_{\e_n}^{\rho}\bar{\xi}_{n}^{(i)}(r_i)^2\,dr_i \notag \\
\le &\ -\frac{1}{8}\sum_{i = 1}^2c_i^2\liminf_{n \to \infty}\frac{1}{|\log \e_n|}\left(\int_{\e_n}^{\rho/2}r_i^{-1}\,dr_i + \int_{\rho/2}^{\rho}\bar{\varphi}(r_i)^2r_i^{-1}\,dr_i\right) \notag \\
= &\ -\frac{1}{8}\sum_{i = 1}^2c_i^2. \label{limsupFen1}
\end{align}
The energy expansion (\ref{Fe1Tay}) follows from Theorem 1.2 in \cite{MR1202527}.
\end{proof}

\subsection{An auxiliary variational problem}
In this section we study the functional 
\[
\J_i(w) \coloneqq \int_{\RR^2_+}|\nabla w(\bm{x})|^2\,d\bm{x} + \int_0^1\left(w(x,0)^2 - c_ix^{-1/2}w(x,0)\right)\,dx + \int_1^{\infty}\left(w(x,0) - \frac{c_i}{2}x^{-1/2}\right)^2\,dx
\]
defined in
\[
H \coloneqq \{w \in H^1_{\loc}(\RR^2_+) : w \in H^1(B_R^+(\bm{0})) \text{ for every } R > 0\},
\]
where $w(\cdot,0)$ indicates the trace of $w$ on the positive real axis.
This functional appears in the characterization of the second order $\Gamma$-convergence of $\F_{\e}$ (see (\ref{Ji}), (\ref{spaceH}), (\ref{A_i}), \Cref{2gc-cpt}, and \Cref{2gc}).
\begin{prop}
\label{auxVP}
Let $\J_i$ and $H$ be given as above. Then $A_i \coloneqq \inf\{\J_i(w) : w \in H\} \in \RR$ and there exists $w_i \in H$ such that $\J_i(w_i) = A_i$. Furthermore, $w_i$ is a weak solution to the mixed problem $(\ref{aux})$.
\end{prop}

\begin{proof}
Let $v$ be the function given in polar coordinates by
\[
\bar{v}(r,\theta) \coloneqq
\left\{
\arraycolsep=1.8pt\def\arraystretch{2.4}
\begin{array}{ll}
\displaystyle \frac{c_i}{2\sqrt{r}} & \text{ if } r > 1 \text{ and } 0 < \theta < \pi , \\
\displaystyle \frac{c_i}{2}\sqrt{r} & \text{ if } r \le 1 \text{ and } 0 < \theta < \pi,
\end{array}
\right.
\]
where $(r,\theta)$ are polar coordinates centered at the origin of $\RR^2$ and such that the set $\{(r,0) : r > 0\}$ coincides with the positive real axis. Then $v \in H$ and $\J_i(v) < \infty$, indeed
\[
\J_i(v) = \int_0^{\pi}\int_0^{\infty}r(\partial_r\bar{v})^2\,drd\theta + \int_0^1\left(\bar{v}(r,0) - c_i\bar{v}(r,0)\right)\,dr = \frac{c_i^2(\pi - 3)}{8}.
\]
In turn, this implies that $A_i < \infty$. On the other hand, by \Cref{lemmacpt}, we see that for every $w \in H$,
\begin{align*}
\J_i(w) \ge &\ \int_{\RR^2_+}|\nabla w(\bm{x})|^2\,d\bm{x} + \int_0^1w(x,0)^2\,dx - |c_i|\kappa\left(\int_{B_1^+(\bm{0})}|\nabla w|^2\,d\bm{x}\right)^{1/2} \\
&\ \quad - |c_i|\kappa\left(\int_0^1w(x,0)^2\,dx\right)^{1/2} + \int_1^{\infty}\left(w(x,0) - \frac{c_i}{2}x^{-1/2}\right)^2\,dx,
\end{align*}
and so $A_i > - \infty$. Furthermore, we deduce that for an infimizing sequence it must be the case that (eventually extracting a subsequence which we don't relabel)
\begin{align*}
\nabla w_n \rightharpoonup &\ \nabla w \quad \quad \quad \quad \quad \quad \ \ \text{ in } L^2(\RR^2_+;\RR^2), \\
w_n(\cdot, 0) \rightharpoonup &\ w(\cdot, 0) \quad \quad \quad \quad \quad \ \, \text{ in } L^2((0,1) \times \{0\}), \\
w_n(\cdot, 0) - \frac{c_i}{2}x^{-1/2} \rightharpoonup &\ w(\cdot, 0) - \frac{c_i}{2}x^{-1/2} \quad \text{ in } L^2((1,\infty) \times \{0\}),
\end{align*}
for some $w \in H$, where $w_n(\cdot,0)$ and $w(\cdot,0)$ indicate the trace of $w_n$ and $w$ on the positive real axis. To conclude, it is enough to show that $\J_i$ is lower semicontinuous for sequences converging as above. The lower semicontinuity is certainly true for the nonnegative terms in $\J_i$, thanks to Fatou's lemma. In order to pass to the limit in the remaining term we can argue as follows. First, we observe that by \Cref{equivnorm} $\{w_n\}_n$ in bounded in $H^1(B_1^+(\bm{0}))$ and in particular in $H^{1/2}((0,1) \times \{0\})$. Next, we recall that $H^{1/2}((0,1) \times \{0\})$ embeds continuously into $L^p((0,1) \times \{0\})$ for every $p \in [1,\infty)$. Consequently, up to the extraction of a further subsequence, we can assume that $w_n \rightharpoonup w$ in $L^p((0,1) \times \{0\})$, $p > 2$. Therefore, we deduce that 
\[
\liminf_{n \to \infty}\int_0^1x^{-1/2}w_n(x,0)\,dx = \int_0^1x^{-1/2}w(x,0)\,dx.
\]	
This proves the existence of a global minimizer of $\J_i$ in $H$. The rest of proposition follows by considering variations of the functional $\J_i$; we omit the details.
\end{proof}

We remark that $w_i$ doesn't necessarily belong to the space $L^2(\RR^2_+)$, unless $c_i = 0$, in which case $w_i \equiv 0$. In the following lemma we prove an estimate on the $L^2$-norm of global minimizers in an annulus that escapes to infinity. This estimate will be crucial for the construction of the recovery sequence for $u_0$ in the proof of \Cref{2gc}.

\begin{lem}
\label{e_nL2}
Let  $\e_n \to 0^+$ and $w_i$ be given as in \Cref{auxVP}. Then
\[
\e_n^2\int_{B_{\rho/\e_n}^+(\bm{0}) \setminus B_{\rho/2\e_n}^+(\bm{0})}w_i^2\,d\bm{x} \to 0
\]
as $n \to \infty$.
\end{lem}

\begin{proof}
By applying \Cref{equivnorm} and by a rescaling argument in $B_1^+(\bm{0}) \setminus B_{1/2}^+(\bm{0})$ we can deduce that there exists a constant $c$, independent of $n$, such that 
\[
\int_{B_{\rho/\e_n}^+(\bm{0}) \setminus B_{\rho/2\e_n}^+(\bm{0})}w^2\,d\bm{x} \le \frac{c}{\e_n^2}\left(\int_{B_{\rho/\e_n}^+(\bm{0}) \setminus B_{\rho/2\e_n}^+(\bm{0})}|\nabla w|^2\,d\bm{x} + \e_n\int_{\rho/2\e_n}^{\rho/\e_n}w(x,0)^2\,dx\right)
\]
for every $w \in H^1(B_{\rho/\e_n}^+(\bm{0}) \setminus B_{\rho/2\e_n}^+(\bm{0}))$. If we apply the previous inequality to $w = \e_n w_i$ we obtain
\[
\e_n^2\int_{B_{\rho/\e_n}^+(\bm{0}) \setminus B_{\rho/2\e_n}^+(\bm{0})}w_i^2\,d\bm{x} \le c\left(\int_{B_{\rho/\e_n}^+(\bm{0}) \setminus B_{\rho/2\e_n}^+(\bm{0})}|\nabla w_i|^2\,d\bm{x} + \e_n\int_{\rho/2\e_n}^{\rho/\e_n}w_i(x,0)^2\,dx\right).
\]
The first term on the right-hand side vanishes as $n \to \infty$ since $\nabla w_i \in L^2(\RR^2_+;\RR^2)$, and the second term is shown to vanish by the following computation:
\begin{align*}
\e_n \int_{\rho/2\e_n}^{\rho/\e_n}w(x,0)^2\,dx \le &\ 2\e_n\int_{\rho/2\e_n}^{\rho/\e_n}\left(w_i(x,0) - \frac{c_i}{2}x^{-1/2}\right)^2\,dx + 2\e_n\int_{\rho/2\e_n}^{\rho/\e_n}\frac{c_i^2}{4x}\,dx \\
= &\ 2\e_n\int_{\rho/2\e_n}^{\rho/\e_n}\left(w_i(x,0) - \frac{c_i}{2}x^{-1/2}\right)^2\,dx + 2\e_n\log2 \to 0
\end{align*}
since $w_i(\cdot,0) - \frac{c_i}{2}x^{-1/2} \in L^2((1,\infty))$. This concludes the proof.
\end{proof}

\subsection{Mixed boundary conditions: Gamma-convergence of order two}
In this section we prove \Cref{2gc-cpt} and \Cref{2gc}. We recall that we use the notations (\ref{origin}) and (\ref{polarxi}).
\begin{proof}[Proof of \Cref{2gc-cpt}]
\textbf{Step 1: }By \Cref{convofmin} we have that $w_n \to u_0$ in $H^1(\Omega)$. For every $n \in \NN$, let $s_n \in L^2(\Omega)$ be such that 
\begin{equation}
\label{wnsn}
w_n = u_0 + \sqrt{\e_n}s_n.
\end{equation}
Then, by (\ref{Fe}), (\ref{F0}), (\ref{Fe1}), (\ref{F1}), and (\ref{Fe2}), $\F_{\e_n}^{(2)}(w_n)$ can be rewritten as 
\[
\F_{\e_n}^{(2)}(w_n) = \frac{1}{\sqrt{\e_n}}\int_{\Omega} \left(\nabla u_0 \cdot \nabla s_n + f s_n\right)\,d\bm{x} + \frac{1}{2}\int_{\Omega}|\nabla s_n|^2\,d\bm{x} + \frac{1}{2\e_n}\int_{\Gamma_D}s_n^2\, d\mathcal{H}^1 + \frac{|\log \e_n|}{8}\sum_{i = 1}^2c_i^2,
\]
and an application of \Cref{IBP} yields
\begin{align*}
\F_{\e}^{(2)}(w_n) = &\ \frac{1}{\sqrt{\e_n}} \left(\int_{\Gamma_D}\partial_{\nu}u^0_{\reg} s_n\, d\mathcal{H}^1 - \sum_{i = 1}^2 \frac{c_i}{2}\int_0^{\rho}\bar{\varphi}(r_i) r_i^{-1/2}\bar{s}^{(i)}_n(r_i,0)\, dr_i\right) \\
&\ \quad + \frac{1}{2}\int_{\Omega}|\nabla s_n|^2\,d\bm{x} + \frac{1}{2\e_n}\int_{\Gamma_D}s_n^2\, d\mathcal{H}^1 + \frac{|\log \e_n|}{8}\sum_{i = 1}^2c_i^2.
\end{align*}
Using the fact that $|\log \e_n| = \int_{\e_n}^1r^{-1}\,dr$, grouping together the different contributions on $\Gamma_D \cap B_{\e_n}(\bm{x}_i)$, $\Gamma_D \cap (B_{\rho}(\bm{x}_i) \setminus B_{\e_n}(\bm{x}_i))$ and $\Gamma_D \setminus B_{\rho}(\bm{x}_i)$, and completing the squares we obtain 
\begin{align*}
\F_{\e_n}^{(2)}(w_n) = &\ \sum_{i = 1}^2\Bigg\{\frac{1}{2}\int_{\e_n}^{\rho}\left(\frac{\bar{s}^{(i)}_n(r_i,0)}{\sqrt{\e_n}} + \overline{\partial_{\nu}u_{\reg}^0}^{(i)}(r_i,0) - \frac{c_i}{2}\bar{\varphi}(r_i)r_i^{-1/2}\right)^2\,dr_i + B_{i,n}c_i + C_{\varphi}c_i^2 \\
&\ \quad + \int_0^{\e_n}\left(\overline{\partial_{\nu}u_{\reg}^0}^{(i)}(r_i,0)\frac{\bar{s}^{(i)}_n(r_i,0)}{\sqrt{\e_n}} - \frac{c_i}{2}r_i^{-1/2}\frac{\bar{s}^{(i)}_n(r_i,0)}{\sqrt{\e_n}} + \frac{\bar{s}^{(i)}_n(r_i,0)^2}{2\e_n}\right)\,dr_i\Bigg\} \\
&\ \quad + \frac{1}{2}\int_{\Gamma_D \setminus \bigcup_iB_{\rho}(\bm{x}_i)}\left(\frac{s_n}{\sqrt{\e_n}} + \partial_{\nu}u_{\reg}^0\right)^2\,d\mathcal{H}^1 - \frac{1}{2}\int_{\Gamma_D \setminus \bigcup_iB_{\e_n}(\bm{x}_i)}\left(\partial_{\nu}u_{\reg}^0\right)^2\,d\mathcal{H}^1 \\
&\ \quad + \frac{1}{2}\int_{\Omega}|\nabla s_n|^2\,d\bm{x},
\end{align*}
where 
\begin{equation}
\label{bin}
B_{i,n} \coloneqq \frac{1}{2}\int_{\e_n}^{\rho}\bar{\varphi}(r_i)r_i^{-1/2}\overline{\partial_{\nu}u_{\reg}^0}^{(i)}(r_i,0)\,dr_i,
\end{equation}
and $C_{\varphi}$ is given as in (\ref{C_i}). Setting 
\begin{equation}
\label{znsn}
z_n \coloneqq s_n - \sqrt{\e_n}u_1,
\end{equation}
where $u_1$ is the solution to (\ref{mixedu1}), and using the fact that $u_1 = -\partial_{\nu}u_{\reg}^0$ on $\Gamma_D$ we can rewrite the previous expression as
\begin{align}
\label{Fe2final}
\F_{\e_n}^{(2)}(w_n) = &\ \sum_{i = 1}^2\Bigg\{\frac{1}{2}\int_{\e_n}^{\rho}\left(\frac{\bar{z}^{(i)}_n(r_i,0)}{\sqrt{\e_n}} - \frac{c_i}{2}\bar{\varphi}(r_i)r_i^{-1/2}\right)^2\,dr_i + B_{i,n}c_i + C_{\varphi}c_i^2 \notag \\
&\ \quad + \frac{1}{2}\int_0^{\e_n}\left(\frac{\bar{z}^{(i)}_n(r_i,0)^2}{\e_n} - c_ir_i^{-1/2}\frac{\bar{z}^{(i)}_n(r_i,0)}{\sqrt{\e_n}}\right)\,dr_i \Bigg\} + \frac{1}{2}\int_{\Gamma_D \setminus \bigcup_iB_{\rho}(\bm{x}_i)}\frac{z_n^2}{\e_n}\,d\mathcal{H}^1 \notag \\
&\ \quad - \frac{1}{2}\int_{\Gamma_D}\left(\partial_{\nu}u_{\reg}^0\right)^2\,d\mathcal{H}^1 + \frac{1}{2}\int_{\Omega}|\nabla(z_n + \sqrt{\e_n}u_1)|^2\,d\bm{x}.
\end{align}
Notice that all the terms in the previous expression are either positive or independent of $n$, with the only exception of $B_{i,n}c_i$, which converges to $B_ic_i$, and the fourth term on the right-hand side. However, by an application of \Cref{lemmacpt} we get
\[
- \int_0^{\e_n}c_i r_i^{-1/2} \frac{\bar{z}^{(i)}_n(r_i,0)}{\sqrt{\e_n}}\,dr_i \ge -|c_i|\kappa\left(\int_{B_{\e_n}^+(\bm{x}_i)}|\nabla z_n|^2\,d\bm{x}\right)^{1/2} -|c_i|\kappa\left(\int_0^{\e_n}\frac{\bar{z}_n^{(i)}(r_i,0)^2}{\e_n}\right)^{1/2},
\]
and thus (\ref{2cpt1}) and (\ref{2cpt2}) are proved at once.
\newline
\textbf{Step 2: }Let $W_{i,n}$ be as in (\ref{barw}). Then 
\begin{equation}
\label{winzn}
\bar{W}_{i,n}(r_i, \theta_i) = \bar{\varphi}(\e_n r_i)\bar{z}^{(i)}_n(\e_n r_i, \theta_i)
\end{equation} 
by (\ref{wnsn}) and (\ref{znsn}), and thus by a change of variables and the fact that $\bar{\varphi} \equiv 1$ in $[0,\rho/2]$, if $\e_n < \rho/2$,
\[
\int_0^1\left(\bar{W}_{i,n}(s,0)^2 - c_is^{-1/2}\bar{W}_{i,n}(s,0)\right)\,ds = \int_0^{\e_n}\left(\frac{\bar{z}^{(i)}_n(r_i,0)^2}{\e_n} - c_ir_i^{-1/2}\frac{\bar{z}^{(i)}_n(r_i,0)}{\sqrt{\e_n}}\right)\,dr_i.
\]
Similarly, for every $R > 1$ and for every $n$ such that $\e_nR < \rho/2$, we have
\begin{align*}
\int_1^R\left(\bar{W}_{i,n}(s,0) - \frac{c_i}{2}s^{-1/2}\right)^2\,ds = &\ \int_{\e_n}^{\e_nR}\left(\frac{\bar{z}^{(i)}_n(r_i,0)}{\sqrt{\e_n}} - \frac{c_i}{2}r_i^{-1/2}\right)^2\,dr_i, \\
\int_{B_R^+(\bm{0})}|\nabla W_{i,n}|^2\,d\bm{y} = &\ \int_{B_{\e_nR}^+(\bm{x}_i)}|\nabla z_n|^2\,d\bm{x}.
\end{align*}
Hence, in view of (\ref{Fe2final}) 
\begin{align}
\label{MFe2}
M \ge \F_{\e_n}^{(2)}(w_n) \ge &\ \sum_{i = 1}^2\Bigg\{\frac{1}{2}\int_0^1\left(\bar{W}_{i,n}(s,0)^2 - c_is^{-1/2}\bar{W}_{i,n}(s,0)\right)\,ds + B_{i,n}c_i + C_{\varphi}c_i^2 \notag \\
&\ \quad + \frac{1}{2}\int_1^R\left(\bar{W}_{i,n}(s,0) - \frac{c_i}{2}s^{-1/2}\right)^2\,ds + \frac{1}{2}\int_{B_R^+(\bm{0})}|\nabla W_{i,n}|^2\,d\bm{y} \Bigg\} \notag \\
&\ \quad - \frac{1}{2}\int_{\Gamma_D}\left(\partial_{\nu}u_{\reg}^0\right)^2\,d\mathcal{H}^1 + \sqrt{\e_n}\int_{\Omega}\nabla z_n \cdot \nabla u_1\,d\bm{x}.
\end{align}
Since $\{\nabla z_n\}_n$ is bounded in $L^2(\Omega;\RR^2_+)$ (see (\ref{2cpt1})), it follows that
\[
\int_{B_R^+(\bm{0})}|\nabla W_{i,n}|^2\,d\bm{y} + \int_0^1\left(\bar{W}_{i,n}(s,0)^2 - c_is^{-1/2}\bar{W}_{i,n}(s,0)\right)\,ds + \int_1^R\left(\bar{W}_{i,n}(s,0) - \frac{c_i}{2}s^{-1/2}\right)^2\,ds \le c,
\]
for some constant $c > 0$ independent of $n$ and $R$. To conclude, it is enough to send $R \to \infty$.
\end{proof}

\begin{proof}[Proof of \Cref{2gc}]
\textbf{Step 1: }Let $\e_n \to 0^+$ and $\{w_n\}_n$ be a sequence of functions in $L^2(\Omega)$ such that $w_n \to w$ in $L^2(\Omega)$. Reasoning as in the proof of \Cref{0gc}, we can assume without loss of generality that 
\[
\liminf_{n \to \infty} \F_{\e_n}^{(2)}(w_n) = \lim_{n \to \infty} \F_{\e_n}^{(2)}(w_n) < \infty.
\]
In particular, $\F_{\e_n}^{(2)}(w_n) < \infty$ for every $n$ sufficiently large. Let $\{w_{n_k}\}_k$ be the subsequence of $\{w_n\}_n$ given in \Cref{2gc-cpt} and for every $k \in \NN$ let $z_{n_k}$ be such that $w_{n_k} = u_0 + \sqrt{\e_{n_k}} z_{n_k} + \e_{n_k} u_1$. Let $W_{i,n}$ be given as in (\ref{winzn}), then by (\ref{Fe2final}), taking $n = n_k$ in (\ref{MFe2}) and letting $k \to 0$ we obtain
\begin{align*}
\liminf_{k \to \infty}\F_{\e_{n_k}}^{(2)}(w_{n_k}) \ge &\ \sum_{i = 1}^2\Bigg\{\frac{1}{2}\int_0^1\left(\bar{W}_i(s,0)^2 - c_is^{-1/2}\bar{W}_i(s,0)\right)\,ds + B_ic_i + C_{\varphi}c_i^2 \\
&\ \quad + \frac{1}{2}\int_1^R\left(\bar{W}_i(s,0) - \frac{c_i}{2}s^{-1/2}\right)^2\,ds + \frac{1}{2}\int_{B_R^+(\bm{0})}|\nabla W_i|^2\,d\bm{y} \Bigg\}\\
&\ \quad - \frac{1}{2}\int_{\Gamma_D}\left(\partial_{\nu}u_{\reg}^0\right)^2\,d\mathcal{H}^1, 
\end{align*}
where we have used (\ref{barw1}), (\ref{barw2}), (\ref{barw3}), and the fact that $\{\nabla z_n\}_n$ is bounded in $L^2(\Omega;\RR^2_+)$ (see (\ref{2cpt1})). By letting $R \to \infty$ in the previous inequality we get
\[
\liminf_{n \to \infty}\F_{\e_n}^{(2)}(w_n) = \lim_{k \to \infty}\F_{\e_{n_k}}^{(2)}(w_{n_k}) \ge \sum_{i = 1}^2\left\{\frac{\J_i(W_i)}{2} + B_ic_i + C_{\varphi}c_i^2\right\} - \frac{1}{2}\int_{\Gamma_D}\left(\partial_{\nu}u_{\reg}^0\right)^2\,d\mathcal{H}^1 \\
\ge \F_2(w),
\]
where in the last step we used the fact that $\J_i(W_i) \ge A_i$.
\newline
\textbf{Step 2: }For every $w \in L^2(\Omega) \setminus \{u_0\}$, the constant sequence $w_n = w$ is a recovery sequence. On the other hand, if $w = u_0$, let $w_i \in H$ be given as in \Cref{auxVP}. Let $z_n$ be the function defined in $B_{\rho}(\bm{x}_i) \cap \Omega$ using polar coordinates around $\bm{x}_i$ (see (\ref{polarxi})) via
\begin{equation}
\label{znpolar}
\bar{z}^{(i)}_n(r_i,\theta_i) \coloneqq \bar{\varphi}(r_i)\bar{W}_i\left(\frac{r_i}{\e_n},\theta_i\right)
\end{equation}
and $z_n(\bm{x}) \coloneqq 0$ in $\Omega \setminus \bigcup_{i = 1}^2B_{\rho}(\bm{x}_i)$. Set
\[
w_n \coloneqq u_0 + \sqrt{\e_n}z_n + \e_n u_1.
\]
We claim that $\{w_n\}_n$ is a recovery sequence for $u_0$. To prove the claim, we notice that (\ref{Fe2final}) implies
\begin{align}
\label{limsupFe2}
\limsup_{n \to \infty}\F_{\e_n}^{(2)}(w_n) \le &\ \sum_{i = 1}^2\Bigg\{\limsup_{n \to \infty} \frac{1}{2}\int_0^{\e_n}\left(\frac{\bar{W}_i(r_i/\e_n,0)^2}{\e_n} - c_ir_i^{-1/2}\frac{w\bar{W}_i(r_i/\e_n,0)}{\sqrt{\e_n}}\right)\,dr_i + B_ic_i + C_{\varphi}c_i^2 \notag \\
&\ \quad + \limsup_{n \to \infty} \frac{1}{2}\int_{\e_n}^{\rho}\varphi_i(r)^2\left(\frac{w_i(r/\e_n,0)}{\sqrt{\e_n}} - \frac{c_i}{2}r^{-1/2}\right)^2\,dr \Bigg\} \notag \\
&\ \quad - \frac{1}{2}\int_{\Gamma_D}\left(\partial_{\nu}u_{\reg}^0\right)^2\,d\mathcal{H}^1 + \limsup_{n \to \infty}\frac{1}{2}\int_{\Omega}|\nabla(z_n + \sqrt{\e_n} u_1)|\,d\bm{x}.
\end{align}
Letting $r = s\e_n$, we obtain
\begin{equation}
\label{VI}
\int_0^{\e_n}\left(\frac{\bar{W}_i(r_i/\e_n,0)^2}{\e_n} - c_i r_i^{-1/2}\frac{\bar{W}_i(r_i/\e_n,0)}{\sqrt{\e_n}}\right)\,dr_i = \int_0^1\left(\bar{W}_i(s,0)^2 - c_is^{-1/2}\bar{W}_i(s,0)\right)\,ds,
\end{equation}
and similarly
\begin{align}
\int_{\e_n}^{\rho}\varphi_i(r)^2\left(\frac{\bar{W}_i(r_i/\e_n,0)}{\sqrt{\e_n}} - \frac{c_i}{2}r_i^{-1/2}\right)^2\,dr = &\ \int_1^{\rho/\e_n}\varphi_i(s\e_n)^2\left(\bar{W}_i(s,0) - \frac{c_i}{2}s^{-1/2}\right)^2\,ds \notag \\
\le &\ \int_1^{\infty}\left(\bar{W}_i(s,0) - \frac{c_i}{2}s^{-1/2}\right)^2\,ds. \label{VII}
\end{align}
Next, we compute the contribution to the energy coming from the gradient term. Since $\bar{\varphi} = 0$ outside of $[0,\rho]$, by (\ref{znpolar}) we have
\begin{align*}
\int_{\Omega}|\nabla z_n|^2\,d\bm{x} = &\ \sum_{i = 1}^2\int_{B_{\rho}(\bm{x}_i)}|\nabla z_n|^2\,d\bm{x} \\
= &\ \sum_{i = 1}^2\int_0^{\pi}\int_0^{\rho}\left[r_i\left(\partial_{r_i}(\bar{\varphi}(r_i)\bar{W}_i(r_i/\e_n,\theta_i)\right)^2 + \frac{1}{r_i}\bar{\varphi}(r_i)^2\left(\partial_{\theta_i}\bar{W}_i(r_i/\e_n,\theta_i)\right)^2\right]\,dr_id\theta_i.
\end{align*}
We write
\[
\int_0^{\pi}\int_0^{\rho}r\left(\partial_r(\varphi_i(r)w_i(r/\e_n,\theta)\right)^2\,drd\theta = \int_0^{\pi}\int_0^{\rho}r\left(\varphi_i'(r)w_i(r/\e_n,\theta) +  \varphi_i(r)\e_n\partial_r w_i(r/\e_n,\theta)\right)^2\,drd\theta.
\]
Expanding the square on the right-hand side of the previous identity we obtain three terms, which we study separately. By the change of variables $s = r_i/\e_n$ we obtain 
\begin{align*}
\int_0^{\pi}\int_0^{\rho}r_i\bar{\varphi}'(r_i)^2\bar{W}_i(r_i/\e_n,\theta_i)^2\,dr_id\theta_i = &\ \int_0^{\pi}\int_0^{\rho/\e_n}s\e_n^2\varphi_i'(s\e_n)^2\bar{W}_i(s,\theta_i)^2\,drd\theta_i \\
\le &\ \frac{c}{\rho}\int_0^{\pi}\int_{\rho/2\e_n}^{\rho/\e_n}s\e_n^2\bar{W}_i(s,\theta)^2\,dsd\theta \to 0,
\end{align*}
where in the last step we have used \Cref{e_nL2}. Similarly,
\begin{align*}
\int_0^{\pi}\int_0^{\rho}r_i\bar{\varphi}(r_i)^2(\partial_{r_i}\bar{W}_i(r_i/\e_n,\theta_i))^2\,dr_id\theta_i = &\ \int_0^{\pi}\int_0^{\rho/\e_n}s\bar{\varphi}(s\e_n)^2(\partial_s\bar{W}_i(s,\theta_i))^2\,dsd\theta_i \\ 
\le &\ \int_0^{\pi}\int_0^{\rho/\e_n}s(\partial_s\bar{W}_i(s,\theta))^2\,dsd\theta.
\end{align*}
In turn, H\"older's inequality implies that
\[
2\int_0^{\pi}\int_0^{\rho}r_i\bar{\varphi}'(r_i)\bar{W}_i(r_i/\e_n,\theta_i)\bar{\varphi}(r_i)\partial_{r_i}\bar{W}_i(r_i/\e_n,\theta_i)\,dr_id\theta_i \to 0
\]
as $n \to \infty$. The same change of variables $s = r_i/\e_n$ also yields
\begin{align*}
\int_0^{\pi}\int_0^{\rho}\frac{\bar{\varphi}(r_i)}{r_i}\left(\partial_{\theta_i}\bar{W}_i(r_i/\e_n,\theta_i)\right)^2\,dr_id\theta_i = &\ \int_0^{\pi}\int_0^{\rho/\e_n}\frac{1}{s}\bar{\varphi}(s\e_n)^2(\partial_{\theta_i}\bar{W}_i(s,\theta_i))^2\,dsd\theta_i \\
\le &\ \int_0^{\pi}\int_0^{\rho/\e_n}\frac{1}{s}(\partial_{\theta_i}\bar{W}_i(s,\theta))^2\,dsd\theta.
\end{align*}
Thus 
\begin{equation}
\label{ineqgrad}
\limsup_{n \to \infty}\int_{\Omega}|\nabla(z_n + \sqrt{\e_n} u_1)|^2\,d\bm{x} \le \limsup_{n \to \infty}\int_{\Omega}|\nabla z_n|^2\,d\bm{x} \le \sum_{i = 1}^2\int_{\RR^2_+}|\nabla W_i|^2\,d\bm{x},
\end{equation}
which, together with (\ref{limsupFe2}), (\ref{VI}), and (\ref{VII}), concludes the proof of the $\Gamma$-limsup inequality.

The energy expansion (\ref{Fe2Tay}) follows from Theorem 1.2 in \cite{MR1202527}.
\end{proof}

\subsection{Sharp estimates}
\begin{proof}[Proof of \Cref{mixedest}]
Suppose by contradiction that (\ref{re-bdry}) is not true. Then there exists a sequence $\e_n \to 0^+$ such that 
\begin{equation}
\label{re-contr}
\|u_{\e_n} - u_0\|_{L^2(\Gamma_R)} > n \left(\e_n \sqrt{|\log \e_n|}\right)
\end{equation}
for every $n \in \NN$. In view of (\ref{Fe1Tay}), we have that 
\[
\sup\{\F_{\e_n}^{(1)}(u_{\e_n}) : n \in \NN\} < \infty,
\]
and thus by \Cref{1gc-cpt} there exist a subsequence $\{u_{\e_{n_k}}\}_k$ of $\{u_{\e_n}\}_n$ and $v_0 \in L^2(\Gamma_D)$ such that
\[
\frac{u_{\e_n} - u_0}{\e_n \sqrt{|\log \e_n|}} \rightharpoonup v_0,
\]
which is a contradiction to (\ref{re-contr}).

The proof of (\ref{re-grad}) follows analogously from (\ref{2cpt1}) and (\ref{Fe2Tay}).
\end{proof}

\section{More general Gamma-convergence results}
Our results can be recast in a more general framework by decoupling the different scales in the asymptotic expansion of $u_{\e}$. Here we present in full detail the generalizations of \Cref{1gc} and \Cref{2gc}; the results of Section 3 can be analogously reformulated. Throughout the section we assume that the domain $\Omega$ is given as in \Cref{H2dec} and use the notations introduced in (\ref{origin}) and (\ref{polarxi}).

\begin{thm}
Under the assumptions of \Cref{1gc-cpt}, let $\K_{\e}^{(1)} \colon L^2(\Omega) \times L^2(\Gamma_D) \to \overline{\RR}$ be defined via
\begin{equation}
\label{Ke1}
\K_{\e}^{(1)}(u,v) \coloneqq 
\left\{
\arraycolsep=1.4pt\def\arraystretch{1.6}
\begin{array}{ll}
\F_{\e}^{(1)}(u) & \text{ if } u \in H^1(\Omega) \text{ and } \frac{u - u_0}{\e \sqrt{|\log \e|}} = v \text{ on } \Gamma_D,\\
+ \infty & \text{ otherwise}.
\end{array}
\right.
\end{equation}
Then the family $\{\K_{\e}^{(1)}\}_{\e}$ $\Gamma$-converges in $L^2(\Omega) \times L^2(\Gamma_D)$ to the functional 
\[
\K_1(u,v) \coloneqq 
\left\{
\arraycolsep=1.4pt\def\arraystretch{1.6}
\begin{array}{ll}
\displaystyle \frac{1}{2}\int_{\Gamma_D}v^2\,d\mathcal{H}^1 - \frac{1}{8} \sum_{i = 1}^2c_i^2 & \text{ if } u = u_0 \text{ and } v \in L^2(\Gamma_D), \\
+ \infty & \text{ otherwise},
\end{array}
\right.
\]
where the coefficients $c_i$ are as in \Cref{H2dec}. 
\end{thm}

\begin{proof}
\textbf{Step 1: }(\emph{Compactness}) Let $\e_n \to 0^+$ and $(u_n,v_n) \in L^2(\Omega) \times L^2(\Gamma_D)$ such that 
\[
\sup \{\K_{\e_n}^{(1)}(u_n,v_n) : n \in \NN\} < \infty.
\]
Then by (\ref{Ke1}), $u_n \in H^1(\Omega)$, the function
\[
v_n^* \coloneqq \frac{u_n - u_0}{\e_n\sqrt{|\log \e_n|}} 
\]
belongs to $H^1(\Omega)$ and satisfies $v_n^* = v_n$ on $\Gamma_D$ in the sense of traces. By \Cref{1gc-cpt}, there exist a subsequence $\{u_{n_k}\}_k$ of $\{u_n\}_n$, $r \in H^1(\Omega)$ and $v \in L^2(\Gamma_D)$ such that
\begin{align*}
\e_{n_k}^{1/2}\nabla v_{n_k}^* \rightharpoonup &\ r \quad \text{ in } H^1(\Omega), \\
v_{n_k} \rightharpoonup &\ v \quad \text{ in } L^2(\Gamma_D).
\end{align*}
\textbf{Step 2: }(\emph{Liminf inequality}) Let $\e_n \to 0^+$ and $\{(u_n,v_n)\}_n$ be a sequence in $L^2(\Omega) \times L^2(\Gamma_D)$ such that $(u_n,v_n) \to (u,v)$. Reasoning as in the proof of \Cref{0gc}, we can assume without loss of generality that 
\[
\liminf_{n \to \infty} \K_{\e_n}^{(1)}(u_n,v_n) = \lim_{n \to \infty} \K_{\e_n}^{(1)}(u_n,v_n) < \infty.
\]
In particular, $\K_{\e_n}^{(1)}(u_n,v_n) < \infty$ for every $n$ sufficiently large. Let $\{u_{n_k}\}_k$ be the subsequence of $\{u_n\}_n$ given as in the previous step and $\xi^i_n$ be the function defined in polar coordinates as in (\ref{xiin}). Then
\[
\liminf_{k \to \infty}\K_{\e_{n_k}}^{(1)}(u_{n_k},v_{n_k}) = \liminf_{k \to \infty}\F_{\e_{n_k}}^{(1)}(u_{n_k})
\]
and so, reasoning as in the proof of \Cref{1gc} (by (\ref{Fenk1}) and (\ref{gammalininf1}) with $v_{n_k}$ and $z_{n_k}$ replaced by $u_{n_k}$ and $v_{n_k}^*$, respectively), we obtain
\begin{align*}
\liminf_{k \to \infty}\K_{\e_{n_k}}^{(1)}(u_{n_k},v_{n_k}) \ge &\ \liminf_{k \to \infty} \left\{ \frac{1}{2}\int_{\Gamma_D}v_{n_k}^2\,d\mathcal{H}^1 - \int_{\Gamma_D \setminus \bigcup_iB_{\e_{n_k}}(\bm{x}_i)}v_{n_k}(\xi_{n_k}^1 + \xi_{n_k}^2)\,d\mathcal{H}^1\right\} \\
\ge &\ \liminf_{k \to \infty}\int_{\Gamma_D \setminus \bigcup_iB_{\e_{n_k}}(\bm{x}_i)}\left[\frac{1}{2}v_{n_k}^2 - v_{n_k}(\xi_{n_k}^1 + \xi_{n_k}^2)\right]\,d\mathcal{H}^1 \\
= &\ \liminf_{k \to \infty}\frac{1}{2}\int_{\Gamma_D \setminus \bigcup_iB_{\e_{n_k}}(\bm{x}_i)}\left[\left(v_{n_k} - \xi_{n_k}^1 - \xi_{n_k}^2\right)^2 - (\xi_n^1)^2 - (\xi_{n_k}^2)^2\right] \,d\mathcal{H}^1 \\
\ge &\ \frac{1}{2}\int_{\Gamma_D}v^2\,d\mathcal{H}^1 - \frac{1}{8} \sum_{i = 1}^2c_i^2 = \K_1(u_0,v),
\end{align*}
where in the last step we have used the fact that $v_{n_k} \rightharpoonup v$, $\xi_{n_k}^i \rightharpoonup 0$ in $L^2(\Gamma_D)$, and so
\[
\liminf_{k \to \infty}\int_{\Gamma_D \setminus \bigcup_iB_{\e_{n_k}}(\bm{x}_i)}\left(v_{n_k} - \xi_{n_k}^1 - \xi_n^2\right)^2\,d\mathcal{H}^1 \ge \int_{\Gamma_D}v^2 \,d\mathcal{H}^1.
\]
\textbf{Step 3: }(\emph{Limsup inequality}) Let $u = u_0$ and $v \in L^2(\Gamma_D)$. We extend $v$ to zero in $\partial \Omega \setminus \Gamma_D$ and assume first that $v \in H^{1/2}(\partial \Omega)$ (in what follows, although with a slight abuse of notation, we identify $v$ with its extension). Then there exists $v^* \in H^1(\Omega)$ such that $v^* = v$ on $\partial \Omega$ in the sense of traces (see Theorem 18.40 in \cite{leoni}). Set 
\[
u_n \coloneqq u_0 + \e_n\sqrt{|\log \e_n|}(z_n + v^*),
\]
where $z_n$ is defined as in (\ref{recFe1}). As one can check (see (\ref{Fen1vnsup}) and (\ref{limsupFen1})), $\{(u_n,z_n + v^*)\}_n$ is a recovery sequence for $(u_0,v)$. 

If $v \in L^2(\partial \Omega) \setminus H^{1/2}(\partial \Omega)$ we consider a sequence $\{v_n\}_n$ of functions in $H^{1/2}(\partial \Omega)$ such that
\begin{equation}
\label{molli1}
\|v_n - v\|_{L^2(\partial \Omega)} \to 0 \quad \text{ as } n \to \infty, 
\end{equation}
and for every $n \in \NN$ we let $v_n^* \in H^1(\Omega)$ be such that $v_n^* = v_n$ on $\partial \Omega$ and 
\begin{equation}
\label{h1h12}
\|v_n^*\|_{H^1(\Omega)} \le c \|v_n\|_{H^{1/2}(\partial \Omega)},
\end{equation}
where $c > 0$ is independent of $n$ (see Theorem 18.40 in \cite{leoni}). Furthermore, notice that by a standard mollification argument we can also assume that 
\begin{equation}
\label{molli}
\|\e_n^{1/2} v_n\|_{H^{1/2}(\partial \Omega)} \to 0 \quad \text{ as } n \to \infty.
\end{equation}
Set
\[
u_n \coloneqq u_0 + \e_n\sqrt{|\log \e_n|}(z_n + v_n^*)
\]
and notice that by (\ref{h1h12}) and (\ref{molli}), $\|\e_n^{1/2} \nabla (z_n + v_n^*)\|_{L^2(\Omega;\RR^2)} \to 0$ as $n \to \infty$. Thus, we can proceed as in (\ref{Fen1vnsup}) and (\ref{limsupFen1}).
\end{proof}

\begin{thm}
Under the assumptions of \Cref{2gc-cpt}, let 
\[
\K_{\e}^{(2)} \colon L^2(\Omega) \times L^2_{\loc}(\RR^2_+) \times L^2_{\loc}(\RR^2_+) \times L^2_{\loc}(\Gamma_D) \to \overline{\RR}
\]
be defined via 
\begin{equation}
\label{Ke2}
\K_{\e}^{(2)}(u,v_1,v_2,w) \coloneqq \F_{\e}^{(2)}(u)
\end{equation}
if 
\begin{equation}
\label{Vw}
\left\{
\arraycolsep=1.4pt\def\arraystretch{1.6}
\begin{array}{ll}
\displaystyle u - u_0 - \e u_1 = \sqrt{\e}V_{i,\e} & \text{ in } \Omega \cap B_{\rho}(\bm{x}_i),\\
u - u_0 - \e u_1 = \e w & \text{ on } \Gamma_D \setminus B_{\e}(\bm{x}_i),
\end{array}
\right.
\end{equation}
where the functions $V_{i,\e}$ are defined in polar coordinates by 
\begin{equation}
\label{vie}
\bar{V}_{i,\e}(r_i, \theta_i) \coloneqq \bar{v}_i\left(\frac{r_i}{\e}, \theta_i\right),
\end{equation}
and $\K_{\e}^{(2)}(u,v_1,v_2,w) \coloneqq + \infty$ otherwise. Then the family $\{\K_{\e}^{(2)}\}_{\e}$ $\Gamma$-converges in $L^2(\Omega) \times L^2_{\loc}(\RR^2_+) \times L^2_{\loc}(\RR^2_+) \times L^2_{\loc}(\Gamma_D)$ to the functional 
\[
\K_2(u,v_1,v_2,w) \coloneqq 
\sum_{i = 1}^2\left[\frac{1}{2}\J_i(v_i) + B_ic_i + C_{\varphi}c_i^2\right] + \frac{1}{2}\int_{\Gamma_D}\left[\left(w - \sum_{i = 1}^2c_i\psi_i\right)^2 - \left(\partial_{\nu}u_{\reg}^0 \right)^2\right]\,d\mathcal{H}^1
\]
if $u = u_0$, $v_1, v_2 \in H$, $w - \sum_{i = 1}^2c_i\psi_i \in L^2(\Gamma_D)$, and $\K_2(u,v_1,v_2,w) \coloneqq +\infty$ otherwise,
where $B_i$ and $C_{\varphi}$ are defined as in $(\ref{B_i})$ and $(\ref{C_i})$, respectively.
\end{thm}

\begin{proof} \textbf{Step 1: }(\emph{Liminf inequality}) Let $\e_n \to 0^+$ and $\{(u_n,v_{1,n},v_{2,n}, w_n)\}_n$ be a sequence in $L^2(\Omega) \times L^2_{\loc}(\RR^2_+) \times L^2_{\loc}(\RR^2_+) \times L^2_{\loc}(\Gamma_D)$ such that $(u_n,v_{1,n},v_{2,n}, w_n) \to (u,v_1,v_2,w)$. Let $\bm{u}_n \coloneqq (u_n,v_{1,n},v_{2,n}, w_n)$. Reasoning as in the proof of \Cref{0gc}, we can assume without loss of generality that 
\[
\liminf_{n \to \infty} \K_{\e_n}^{(2)}(\bm{u}_n) = \lim_{n \to \infty} \K_{\e_n}^{(2)}(\bm{u}_n) < \infty.
\]
In particular, $\K_{\e_n}^{(2)}(\bm{u}_n) < \infty$ for every $n$ sufficiently large. Let $\{u_{n_k}\}_k$ be the subsequence of $\{u_n\}_n$ given as in \Cref{2gc-cpt}. By (\ref{Fe2final}) (with $w_n$ replaced by $u_{n_k}$), (\ref{Ke2}), (\ref{Vw}), and (\ref{vie}) it follows that for every $\e_{n_k} < \delta < \rho$, 
\begin{align}
\label{Ke2final}
\K_{\e_{n_k}}^{(2)}(\bm{u}_{n_k}) = &\ \sum_{i = 1}^2\Bigg\{\frac{1}{2}\int_{\e_{n_k}}^{\delta}\left(\frac{\bar{v}_{i,n_k}(r_i/\e_{{n_k}},0)}{\sqrt{\e_{n_k}}} - \frac{c_i}{2}\bar{\varphi}(r_i)r_i^{-1/2}\right)^2\,dr_i + B_{i,n_k}c_i + C_{\varphi}c_i^2 \notag \\
&\ \quad + \frac{1}{2}\int_0^{\e_{n_k}}\left(\frac{\bar{v}_{i,n_k}(r_i/\e_{n_k},0)^2}{\e_{n_k}} - c_ir_i^{-1/2}\frac{\bar{v}_{i,n_k}(r_i/\e_{n_k},0)}{\sqrt{\e_{n_k}}}\right)\,dr_i \Bigg\} \notag \\
&\ \quad + \frac{1}{2}\int_{\Gamma_D \setminus \bigcup_iB_{\delta}(\bm{x}_i)}\left(w_{n_k} - \sum_{i = 1}^2c_i\psi_i\right)^2\,d\mathcal{H}^1 - \frac{1}{2}\int_{\Gamma_D}\left(\partial_{\nu}u_{\reg}^0\right)^2\,d\mathcal{H}^1 \notag \\
&\ \quad + \frac{1}{2\e_{n_k}}\int_{\Omega}|\nabla(u_{n_k} - u_0)|^2\,d\bm{x},
\end{align}
where $B_{i,n_k}$ is defined as in (\ref{bin}). Arguing as in the first step of the proof of \Cref{2gc}, we arrive at
\begin{align*}
\liminf_{k \to \infty}\K_{\e_{n_k}}^{(2)}(\bm{u}_{n_k}) \ge &\ \sum_{i = 1}^2\left[\frac{1}{2}\J_i(v_i) + B_ic_i + C_{\varphi}c_i^2\right] + \frac{1}{2}\int_{\Gamma_D \setminus \bigcup_iB_{\delta}(\bm{x}_i)}\left(w - \sum_{i = 1}^2c_i\psi_i\right)^2\,d\mathcal{H}^1 \\
&\ \quad - \frac{1}{2}\int_{\Gamma_D}\left(\partial_{\nu}u_{\reg}^0 \right)^2\,d\mathcal{H}^1.
\end{align*}
To conclude the proof of the liminf inequality it is enough to let $\delta \to 0^+$.
\newline
\textbf{Step 2: }(\emph{Limsup inequality}) Let $(u_0,v_1,v_2,w)$ be such that $\K_2(u_0,v_1,v_2,w) < \infty$. We assume first that there exists $0 < \delta < \rho/2$ such that 
\begin{equation}
\label{addass}
w \in H^{1/2}\left(\Gamma_D \setminus \bigcup_{i = 1}^2\overline{B_{\delta/4}(\bm{x}_i)}\right),
\end{equation}
and we extend it to a function in $H^{1/2}(\partial \Omega)$ (in what follows, although with a slight abuse of notation, we identify $w$ with its extension). Then there exists $w^* \in H^1(\Omega)$ such that $w^* = w$ on $\partial \Omega$ in the sense of traces (see Theorem 18.40 in \cite{leoni}). Set 
\[
u_n \coloneqq u_0 + \e_n u_1 + \sqrt{\e_n}Z_n, 
\]
where $Z_n$ is given in polar coordinate at $\bm{x}_i$ by
\[
\bar{Z}_n^{(i)}(r_i,\theta_i) \coloneqq \bar{\varphi}\left(\frac{\rho}{2\delta}r_i\right)\bar{v}_i\left(\frac{r_i}{\e_n}, \theta_i\right) + \sqrt{\e_n}\left(1 - \bar{\varphi}\left(\frac{\rho}{2\delta}r_i\right)\right)\overline{w^*}^{(i)}(r_i,\theta_i),
\]
and $Z_n \coloneqq \sqrt{\e_n}w^*$ in $\Omega \setminus \bigcup_{i = 1}^2B_{\rho}(\bm{x}_i)$.
We claim that $\{\bm{u}_n\}_n$, defined from $\{u_n\}_n$ via (\ref{Vw}) and (\ref{vie}), is a recovery sequence for $(u_0,v_1,v_2,w)$. Using the fact that $\bar{\varphi}\left(\frac{\rho}{2\delta}r_i\right) = 1$ for $r_i \le \delta$ and the change of variables $\e_n s = r_i$ (see also (\ref{VI}), (\ref{VII}), and (\ref{ineqgrad})), we get 
\begin{align*}
\J_i(v_i) \ge &\ \limsup_{n \to \infty}\Biggr\{\int_{B_{\delta}(\bm{x}_i)}|\nabla Z_n|^2\,d\bm{x} + \int_0^{\e_n}\left(\frac{\bar{Z}_n^{(i)}(r_i,0)^2}{\e_n} - c_ir_i^{-1/2}\frac{\bar{Z}_n^{(i)}(r_i,0)}{\sqrt{\e_n}}\right)\,dr_i \\
&\ \quad + \int_{\e_n}^{\delta}\left(\frac{\bar{Z}_n^{(i)}(r_i,0)}{\sqrt{\e_n}} - \frac{c_i}{2}\bar{\varphi}(r_i)r_i^{-1/2}\right)^2\,dr_i\Biggr\}.
\end{align*}
In turn, it follows from (\ref{Ke2final}) that
\begin{align}
\limsup_{n \to \infty}\K_{\e}^{(2)}(\bm{u}_n) \le &\ \sum_{i = 1}^2\left\{\frac{\J_i(v_i)}{2} + B_ic_i + C_{\varphi}c_i^2\right\} + \limsup_{n \to \infty}\frac{1}{2}\int_{\Gamma_D \setminus \bigcup_iB_{\delta}(\bm{x}_i)}\left(\frac{Z_n}{\sqrt{\e_n}} - \sum_{i = 1}^2c_i\psi_i\right)^2\,d\mathcal{H}^1 \notag \\
&\ \quad  - \frac{1}{2}\int_{\Gamma_D}\left(\partial_{\nu}u_{\reg}^0\right)^2\,d\mathcal{H}^1 + \limsup_{n \to \infty}\frac{1}{2}\int_{\Omega \setminus \bigcup_i B_{\delta}(\bm{x}_i)}|\nabla(Z_n + \sqrt{\e_n} u_1)|^2\,d\bm{x}. \label{Ke2sup}
\end{align}
By the convexity of the square function we have
\begin{align*}
\int_{\delta}^{2\delta}\left(\frac{\bar{Z}^{(i)}_n(r_i,0)}{\sqrt{\e_n}} - \frac{c_i}{2}\bar{\varphi}(r_i)r_i^{-1/2}\right)^2\,dr_i \le &\ \int_{\delta}^{2\delta}\bar{\varphi}\left(\frac{\rho}{2\delta}r_i\right)\left(\bar{v}_i(r_i/\e_n,0) - \frac{c_i}{2}\bar{\varphi}(r_i)r_i^{-1/2}\right)^2\,dr_i \\
&\ \quad + \int_{\delta}^{2\delta}\left(1 - \bar{\varphi}\left(\frac{\rho}{2\delta}r_i\right)\right)\left(w - \frac{c_i}{2}\bar{\varphi}(r_i)r^{-1/2}\right)^2\,dr_i,
\end{align*}
and therefore, since $\J_i(v_i) < \infty$,
\[
\limsup_{n \to \infty}\int_{\delta}^{2\delta}\left(\frac{\bar{Z}^{(i)}_n(r_i,0)}{\sqrt{\e_n}} - \frac{c_i}{2}\bar{\varphi}(r_i)r_i^{-1/2}\right)^2\,dr_i \le \int_{\delta}^{2\delta}\left(w - \frac{c_i}{2}\bar{\varphi}(r_i)r^{-1/2}\right)^2\,dr_i.
\]
In addition, using the fact that $\bar{\varphi}\left(\frac{\rho}{2\delta}r_i\right) = 0$ for $r_i \ge 2\delta$, we obtain
\[
\int_{\Gamma_D \setminus \bigcup_iB_{2\delta}(\bm{x}_i)}\left(\frac{Z_n}{\sqrt{\e_n}} - \sum_{i = 1}^2c_i\psi_i\right)^2\,d\mathcal{H}^1 = \int_{\Gamma_D \setminus \bigcup_iB_{2\delta}(\bm{x}_i)}\left(w - \sum_{i = 1}^2c_i\psi_i\right)^2\,d\mathcal{H}^1.
\]
We now observe that the result of \Cref{e_nL2} straightforwardly extends to every $v_i \in H$ such that $\J_i(v_i) < \infty$. Consequently, we can argue as in the second step of the proof of \Cref{2gc} to deduce that
\[
\limsup_{n \to \infty}\frac{1}{2}\int_{\Omega \setminus \bigcup_i B_{\delta}(\bm{x}_i)}|\nabla(Z_n + \sqrt{\e_n} u_1)|^2\,d\bm{x} = 0.
\]
This concludes the proof of the limsup inequality under the assumption that (\ref{addass}) is satisfied.

If on the other hand 
\[
w \notin H^{1/2}\left(\Gamma_D \setminus \bigcup_{i = 1}^2\overline{B_{\delta/4}(\bm{x}_i)}\right)
\]
for any $\delta > 0$, we reproduce the mollification argument in (\ref{molli1}) - (\ref{molli}) and proceed as before.
\end{proof}

\section*{Acknowledgements}
This paper is part of the first author's Ph.D. thesis, Carnegie Mellon University. The authors acknowledge the Center for Nonlinear Analysis  where part of this work was carried out. The research of G. Gravina and G. Leoni was partially funded by the National Science Foundation under Grants No. DMS-1412095 and DMS-1714098.

\bibliographystyle{alpha}
\bibliography{mixed}
\end{document}